\documentclass[a4paper,reqno,11pt]{amsart}
\usepackage{mathrsfs}
\usepackage[all]{xy}
\usepackage{txfonts}
\usepackage{amsfonts}
\usepackage{amssymb}
\usepackage{amsfonts,amsmath,amsthm,amssymb,stmaryrd,color}
\usepackage[numbers,sort&compress]{natbib}
\usepackage{hyperref}

\newtheorem{Theorem}{Theorem}[section]
\newtheorem{Lemma}{Lemma}[section]
\newtheorem{Proposition}{Proposition}[section]
\newtheorem{Remark}{Remark}[section]

\numberwithin{equation}{section}
\usepackage[left=1 in, right=1 in,top=1 in, bottom=1 in]{geometry}

\def\XXint#1#2#3{{\setbox0=\hbox{$#1{#2#3}{\int}$ }
\vcenter{\hbox{$#2#3$ }}\kern-.6\wd0}}

\DeclareMathOperator{\diver}{div}

\def\r3{\mathbb{R}^3}

\makeatletter
\@namedef{subjclassname@2020}{\textup{2020} Mathematics Subject Classification}
\makeatother

\allowdisplaybreaks

\begin{document}

\title[ The 2D inhomogeneous incompressible viscoelastic system ]{\bf G\MakeLowercase{lobal solutions of the 2}D \MakeLowercase{inhomogeneous incompressible viscoelastic system}}
\author{Chengfei Ai}
\address{School of Mathematics and Statistics, Yunnan University,  Kunming, Yunnan 650091, China.}
\email[C.F. Ai]{aicf5206@163.com}

\author{Yong Wang$^\ast$}
\address{ School of Mathematical Sciences, South China Normal University, Guangzhou, Guangdong 510631, China.}
\email[Y. Wang]{wangyongxmu@163.com}

\author{Yunshun Wu}
\address{School of Mathematical Sciences, Guizhou Normal University, Guiyang, Guizhou 550001, China.}
 \email[Y.S. Wu]{wuyunshun1979@163.com}

\thanks{$^\ast$Corresponding author: Yong Wang,\ wangyongxmu@163.com}

\begin{abstract}
In this paper, we investigate the global existence of strong solutions for the inhomogeneous incompressible viscoelastic system with only velocity dissipation on $\mathbb{R}^{2}$. Due to the criticality of the time-weight, the methods for the corresponding problem on $\mathbb{R}^{3}$ cannot be directly applied to the two-dimensional case. To overcome the main difficulties, we first transform the original system into a suitable dissipative system by introducing an effective tensor. Then we develop a new fractional time-weighted energy framework, combined with elegant commutator and bilinear estimates, to prove the global existence of strong solutions without the help of the common ``div-curl" structure on the viscoelastic system.
\end{abstract}


\maketitle

\section{Introduction}
Viscoelastic fluids, due to their rich rheological phenomena and a wide range of applications in industry and engineering (e.g., shampoo, ketchup, drilling mud, natural asphalt, paints, biofluids, polymer solutions, etc.), have received a lot of attention in recent years. As a specific type of non-Newtonian fluids, viscoelastic fluids are such materials that behave like both (viscous) liquids and (elastic) solids, which possess many physical properties that are different from Newtonian fluids, such as elastic instability and elastic turbulence at low Reynolds numbers (see e.g., \cite{Groisman-Steinberg2000,Steinberg2019}). In this paper, we are devoted to the analysis of the inhomogeneous incompressible viscoelastic system \eqref{1.1}, which is derived by using an energetic variational approach as done in \cite{Ai-Wang2025} even though it has been directly proposed in \cite{Hu-Wang2009}.

Precisely, we study the Cauchy problem of the two-dimensional (2D) inhomogeneous incompressible viscoelastic system:
\begin{equation}\label{1.1}
\begin{cases}
\tilde{\rho}_{t}+u\cdot\nabla\tilde{\rho}=0,\\
\tilde{\rho} u_{t}+\tilde{\rho}u\cdot\nabla u+\nabla p=\mu\Delta u+c^{2}\diver(\tilde{\rho} \mathbb{F}\mathbb{F}^{\top}),\\
\mathbb{F}_{t}+ u\cdot\nabla \mathbb{F}=\nabla u\mathbb{F}, \\
\diver u=0, \qquad\qquad (x,t)\in\mathbb{R}^{2}\times\mathbb{R}^{+},
\end{cases}
\end{equation}
which is subject to the initial data
\begin{equation}\label{1.1'}
(\tilde{\rho},u,\mathbb{F})(x,t)\mid_{t=0}=(\tilde{\rho}_{0}(x),u_{0}(x),\mathbb{F}_{0}(x))\to (1,0,\mathbb{I})\quad \mbox{as}\quad x\to\infty.
\end{equation}
Here $\tilde{\rho}>0$ is the fluid density, $u\in\mathbb{R}^{2}$ is the velocity field, $\mathbb{F}\in \mathbb{M}^{2\times2}$ (the set of $2\times 2$ matrices with positive determinants) is the deformation gradient of fluids, $p$ is the pressure (the Lagrange multiplier), and $\mathbb{I}$ is the identity matrix. The constants $\mu>0$ and $c>0$ represent the shear viscosity and the speed of elastic wave propagation, respectively. The superscript $\top$ represents the transpose.

Let us first look back on some remarkable results of (homogeneous or inhomogeneous) incompressible and compressible viscoelastic fluids from different perspectives. When the density is constant (say $\tilde{\rho}\equiv1$) in (\ref{1.1}), the reduced system (\ref{1.1}) first proposed by Lin et al. \cite{Lin-Liu-Zhang2005} governs the dynamic behaviors of the homogeneous incompressible viscoelastic fluids. Let $\mathbb{E}=\mathbb{F}-\mathbb{I}$ be the perturbation for the deformation tensor $\mathbb{F}$ with respect to the identity matrix $\mathbb{I}$ as its equilibrium state. For the 2D Cauchy problem, Lin et al. \cite{Lin-Liu-Zhang2005} analyzed an auxiliary vector field with the ``div" structure $\diver\mathbb{E}^{\top}=0$, then they proved the global existence of small strong solutions without additional damping mechanisms or (viscous) dissipative effects. Later, Lei et al. \cite{Lei-Liu-Zhou2008} established the following ``curl" structure
\begin{equation}\label{0.2}
\nabla_{k}\mathbb{E}^{ij}-\nabla_{j}\mathbb{E}^{ik}=\mathbb{E}^{lj}\nabla_{l}\mathbb{E}^{ik}-\mathbb{E}^{lk}\nabla_{l}\mathbb{E}^{ij}.
\end{equation}
Combining the ``div" structure with the ``curl" structure, Lei et al. \cite{Lei-Liu-Zhou2008} proved the existence of global small solutions in 2 and 3 space dimensions. Furthermore, Chen and Zhang \cite{Chen-Zhang2006} employed another ``curl-free" structure $\nabla\times(\mathbb{F}^{-1}-\mathbb{I})=0$ to study the same problem. By virtue of the scaling invariant
approach, Zhang and Fang \cite{Zhang-Fang2012} proved the well-posedness of incompressible viscoelastic fluids in the critical $L^{p}$ framework.
After that, Fang, Zhang and Zi \cite{Fang-Zhang-Zi2018} took the dispersive effect into account, and proved the global well-posenedness of solutions with large initial velocity concentrating on the low frequency part in the critical space $\dot{B}^{\frac{d}{2}-1}_{2,1}$. Hu and Wu \cite{Hu-Wu2015} obtained the weak-strong uniqueness property in the class of finite energy weak solutions for the incompressible viscoelastic system. Moreover, Zhang \cite{Zhang2014}, Fang-Zi \cite{Fang-Zi2016} and Jiang-Jiang \cite{Jiang-Jiang2021} considered a class of large initial data and obtained the global existence of strong solutions. Based on the above ``div-curl" structure, the initial-boundary value
problems in domains with the solid boundary and the free boundary problems were also considered in \cite{Lin-Zhang2008,DiIorio-Marcati-Spirito2020,He-Xu2010,Feng-Zhu-Zi2017}, respectively. For more related research work, please refer to \cite{Feng-Wang-Wu2022,He-Zi2021,Hu-Lin2016,Wang-Wu-Xu-Zhong2022,
Jiang-Jiang-Wu2017,Lai-Lin-Wang2017,Lei-Liu-Zhou2007,Lei-Zhou2005} and the references cited therein.

For the compressible case corresponding to \eqref{1.1}, the following three identities play an essential role in obtaining the global solutions:
\begin{equation}\label{0.3}
\tilde{\rho}\det \mathbb{F}=1;
\end{equation}
\begin{equation}\label{0.4}
\diver (\tilde{\rho}\mathbb{F}^{\top})=0;
\end{equation}
\begin{equation}\label{0.5}
\mathbb{F}^{lk}\nabla_{l}\mathbb{F}^{ij}-\mathbb{F}^{lj}\nabla_{l}\mathbb{F}^{ik}=0.
\end{equation}
Here (\ref{0.4}) and (\ref{0.5}) are also called ``div-curl" structure. And now it is clear that combining one of \eqref{0.3} and \eqref{0.4} with \eqref{0.5} can infer another identity; see \cite{Hu-Zhao2020}.
Using the identities (\ref{0.3})--(\ref{0.5}) (sometimes called the intrinsic properties), Qian and Zhang \cite{Qian-Zhang2010} proved the local large and global small solutions of the Cauchy problem in critical Besov spaces; see also \cite{Hu-Wang2011,Han-Zi2020}. For the initial-boundary value problems, Qian \cite{Qian2011} and Hu et al. \cite{Hu-Wang2015} established the global well-posedness of the strong solution
near its equilibrium state, where the former used energy methods and the later employed the theory of the maximal regularities. Further, Chen and Wu \cite{Chen-Wu2018} proved the exponential decay rates of strong solutions in the $H^{2}$-Sobolev space. Wang et al. \cite{Wang-Shen-Wu-Zhang2022} proved the existence of global-in-time small strong solution for the three-dimensional compressible viscoelastic fluids with the electrostatic effect (where the electrostatic potential is imposed the Dirichlet-Neumann mixed boundary condition) in a bounded domain. Recently, Gu et al. \cite{Gu-Wang-Xie2024} considered the vanishing viscosity limit of compressible viscoelastic equations in the half space. Hu et al. \cite{Hu-Ou-Wang-Yang2023} studied the incompressible limit for compressible viscoelastic flows with large initial velocity. Huang et al. \cite{Huang-Yao-You2025} proved the global well-posedness of the three-dimensional free boundary problem for viscoelastic fluids in an infinite strip without the surface tension. For more related results on the compressible case, we refer to \cite{Hu-Meng-Zhang2025,Hu-Zhao2020,Pan-Xu-Zhu2022,Tan-Wang-Wu2020,Wang-Wu2021,Zhu2020,Wu-Wang2023} and the references therein. By using similar processing methods for homogeneous incompressible and compressible cases, under the identities (\ref{0.3})--(\ref{0.5}), some results can be found in \cite{Fang-Han-Zhang2014,Han2016,Hu-Wang2009,Qiu-Fang2018,
Jiang-Wu-Zhong2016} for the inhomogeneous incompressible viscoelastic system \eqref{1.1}. However, without using the ``div-curl" structural identities \eqref{0.4}-\eqref{0.5}, it is also possible to prove the global well-posedness of the Cauchy problem for the inhomogeneous incompressible viscoelastic system \eqref{1.1}, which is a new and difficult problem. Recently, by treating the wildest ``nonlinear term" as ``linear term" through an elegant time-weighted energy framework, Zhu \cite{Zhu2018,Zhu2022} considered the global existence of small solutions to the homogeneous incompressible and compressible viscoelastic system without using the ``div-curl" structure in the 3D whole space. Later, we extended the results in \cite{Zhu2018,Zhu2022} to the inhomogeneous incompressible system \eqref{1.1} by introducing a new effective tensor to transform the original system into a suitable dissipative one in \cite{Ai-Wang2025}.

Until now, it is unknown that the global well-posedness in Sobolev settings for the inhomogeneous incompressible viscoelastic system \eqref{1.1} without using the ``div-curl" structure in the 2D whole space. Due to the criticality of the time-weight, the processing method of the 3D system in \cite{Ai-Wang2025} cannot be directly extended to the 2D case. The most inspiring work in this area mainly includes: Lin et al. \cite{Lin-Wei-Wu2022} established the global well-posedness of the 2D homogeneous incompressible Oldroyd-B model in a periodic domain, whose results also hold for the homogeneous incompressible case of \eqref{1.1} (see the arguments in \cite{Zhu2018}). Note that the method used in \cite{Lin-Wei-Wu2022} cannot be applied to the 2D whole space due to the failure of Poincar\'{e}'s inequality. Through some carefully designed fractional time-weighted energy framework, Chen and Zhu \cite{Chen-Zhu2023} proved the global existence of small solutions in the Sobolev setting for the homogeneous incompressible viscoelastic system in the 2D whole space. However, it is not enough to prove the global well-posedness of solutions for the inhomogeneous incompressible viscoelastic system \eqref{1.1} by the methods in \cite{Lin-Wei-Wu2022,Chen-Zhu2023}. The main difficulties are as follows: since the density in the system \eqref{1.1} is not constant at all, the velocity equation in \eqref{1.1} is a parabolic equation with the variable coefficient, and thus the pressure term $\nabla p$ cannot be directly eliminated through the Helmholtz projection operator. At the same time, it is very tricky to capture the damping mechanism of density $\tilde{\rho}$ and  deformation $\mathbb{F}$. In order to overcome these difficulties mentioned above, we need to develop a more delicate fractional time-weighted energy framework as done in this paper.

The novelty of this article mainly lies in the following several aspects. Firstly, how can we deal with $\|u\|_{L^{\infty}}\|\nabla\rho\|_{L^{2}}\|\rho\|_{L^{2}}\lesssim\|\nabla u\|_{\dot{H}^{s_{1}-s_{0}}}^{1-\theta}\|\nabla u\|_{\dot{H}^{2}}^{\theta}\|\nabla\rho\|_{L^{2}}\|\rho\|_{L^{2}}$  in (\ref{2.10}) so as to close the assistant energy $\mathcal{E}_{a}(t)$? Due to the unboundedness of the integral $\int_{0}^{\infty}(1+t)^{-1}dt=\infty$, it is usually very difficult to obtain the dissipation estimate of $\|\nabla u\|_{\dot{H}^{s_{1}-s_{0}}}^{1-\theta}\|\nabla u\|_{\dot{H}^{2}}^{\theta}$. For this purpose, by virtue of a more refined selection of weighted indices $s_{0}$ and $s_{1}$ as stated in the following
\begin{align*}
\frac{\sqrt{7}-1}{3}<s_{1}<1,\quad \max\Big\{s_{1},\frac{3s_{1}^{2}-2s_{1}+2}{3s_{1}}\Big\}<s_{0}<\min\{1+s_{1},2s_{1}\}
\end{align*}
such that $\frac{(1+2s_{1}-s_{0})}{2-\theta}>1$, from which it can be derived $\int_{0}^{\infty}(1+t)^{-\frac{(1+2s_{1}-s_{0})}{2-\theta}}dt<\infty$, then we can establish  the estimate of (\ref{2.11}). Similarly, we can also obtain estimates of (\ref{2.17}) and (\ref{2.53}).

Secondly, when estimating the basic energy $\mathcal{E}_{0}(t)$, the main difficulty is to deal with the pressure term $\frac{1}{\rho+1}\nabla p$ in the first equation of system (\ref{2.4}). Our idea is to transform the pressure $p$ into $\tilde{p}$ by $\frac{1}{\rho+1}\nabla p:=\nabla\tilde{p}$, so that we can overcome the pressure term $\nabla p$ which cannot be directly eliminated through the Helmholtz projection operator.

Lastly, to get the estimate of the fractional time-weighted energy $\mathcal{E}_{1}(t)$ as stated in Lemma \ref{le2.5}, several new and essential difficulties appear in the later analysis since the density is not constant. In particular, it is hard to estimate the terms $J_{5}$ in \eqref{2.52} and $J_{10}$ in \eqref{2.69}. To control these two terms $J_{5}$ and $J_{10}$ effectively, we usually need the dissipation estimation of $\|g(\rho)\diver G\|_{\dot{H}^{s_{1}-s_{0}}}\|\nabla u\|_{\dot{H}^{1+s_{1}-s_{0}}}$, which seems impossible to achieve by the previous method in \cite{Chen-Zhu2023}. Fortunately, by defining the appropriate assistant energy $\mathcal{E}_{a}(t)$ (which includes $\|\rho(t')\|_{\dot{H}^{s_{1}-s_{0}}}$), we can prove that $\|\rho(t')\|_{\dot{H}^{s_{1}-s_{0}}}$ is bounded in (\ref{2.17}). Moreover, we derive a new bilinear estimate in Lemma \ref{le5.4}, which is an extension of Proposition 2.4 in  \cite{Chen-Zhu2023}. Using the assistant energy $\mathcal{E}_{a}(t)$ and Lemma \ref{le5.4}, we finally manage to obtain the estimates of $J_{5}$ and $J_{10}$, namely, \eqref{2.64} and \eqref{2.83}.

\medskip

\textbf{Notation.} Throughout this paper, we use $a\lesssim b$ to denote $a\leq C b$ and $a\gtrsim b$ to denote $a\geq C b$ for a generic constant $C>0$. The relation $a\sim b$ represents $a\lesssim b$ and $a\gtrsim b$. Except for special emphasis, we let $C$ denote a universal positive constant. Let $\nabla^{k}=\partial_{x}^{k}$ with an integer $k\geq0$ be the usual spatial derivatives of order $k$. Moreover, for $s<0$ or $s$ is not a positive integer, $\nabla^{s}$ stands for $\Lambda^{s}$, that is,
\begin{equation*}
\nabla^{s}f=\Lambda^{s}f:=\mathscr{F}^{-1}(|\xi|^s\mathscr{F}f),
\end{equation*}
where $\mathscr{F}$ is the usual Fourier transform operator and $\mathscr{F}^{-1}$ is its inverse (see e.g., \cite{Bahouri-Chemin-Danchin2011}). We use $\dot{H}^s(\mathbb{R}^n)$ $(s\in\mathbb{R})$ to denote the homogeneous Sobolev spaces on $\mathbb{R}^n$ with the norm $\|\cdot\|_{\dot{H}^s}$ defined by $\|f\|_{\dot{H}^s}:=\|\Lambda^sf\|_{L^2}$. And $H^s(\mathbb{R}^n)$ and $L^{p}(\mathbb{R}^n)$ are the usual Sobolev and Lebesgue spaces with the norm $\|\cdot\|_{H^s}$ and the norm $\|\cdot\|_{L^{p}}$, respectively. For simplicity, we do not distinguish functional spaces when scalar-valued or vector-valued functions are involved.

Now we are in a position to present the main result.

\begin{Theorem}\label{th1.1}
Suppose that the initial data $(\tilde{\rho}_{0}, u_{0}, \mathbb{F}_{0})$ with $\diver u_0=0$ satisfy for some sufficiently small constant $\varepsilon>0$,
\begin{equation}
\|u_{0}\|_{\dot{H}^{-s_{1}}\cap\dot{H}^{2}}+\|\tilde{\rho}_{0}-1\|_{H^{2}\cap\dot{H}^{s_{1}-s_{0}}}+\|\mathbb{F}_{0}-\mathbb{I}\|_{H^{2}}+\| G_{0}\|_{\dot{H}^{-s_{1}}\cap\dot{H}^{2}}\leq\varepsilon,
\end{equation}
where $s_{0}, s_{1}$ are any given constants satisfying the bounds
\begin{align}\label{s-0s-1}
\frac{3s_{1}^{2}-2s_{1}+2}{3s_{1}}<s_{0}<2s_{1},\quad \frac{\sqrt{7}-1}{3}\approx 0.5486<s_{1}<1.
\end{align}
Then the Cauchy problem \eqref{1.1}--\eqref{1.1'} admits a unique global solution $(\tilde{\rho}, u, \mathbb{F})(t)$ such that
\begin{align*}
\sup_{0\leq t\leq \infty}&\big[\|u(t)\|_{\dot{H}^{-s_{1}}\cap\dot{H}^{2}}^{2}+\|(\tilde{\rho}-1)(t)\|_{H^{2}\cap\dot{H}^{s_{1}-s_{0}}}^{2}+\|(\mathbb{F}-\mathbb{I})(t)\|_{H^{2}}^{2}+\| G(t)\|_{\dot{H}^{-s_{1}}\cap\dot{H}^{2}}^{2}\big]\\
&+\int_{0}^{\infty}(1+t)^{1+2s_{1}-s_{0}}\|\nabla u(t)\|_{\dot{H}^{1+s_{1}-s_{0}}\cap\dot{H}^{2}}^{2}\,dt\leq\varepsilon,
\end{align*}
where $G:=\tilde{\rho}\mathbb{F}\mathbb{F}^{\top}-\mathbb{I}$.
\end{Theorem}

We give some remarks in the following.

\begin{Remark}
According to the definitions of various energies in section 2, the small assumption about the initial data $\|(u_{0},G_{0})\|_{\dot{H}^{-s_{1}}}+\|\tilde{\rho}_{0}-1\|_{\dot{H}^{s_{1}-s_{0}}}<\varepsilon$ in Theorem \ref{th1.1} cannot be removed. Moreover, as mentioned in \cite{Chen-Zhu2023}, if the initial data $(\tilde{\rho}_{0}, u_{0}, \mathbb{F}_{0})$ only belong to the non-negative
index Sobolev space $H^{2}(\mathbb{R}^{2})$, the existence of global classical solutions for the inhomogeneous incompressible viscoelastic system \eqref{1.1} remains unknown without help of the ``div-curl" structure.
\end{Remark}

\begin{Remark}
Actually, the precise selection of $s_{0}$ and $s_{1}$ are effective for enclosing assistant energy $\mathcal{E}_{a}(t)$. And the range of values for $s_{0}$ and $s_{1}$ are also non-empty (e.g., $s_{1}=0.549, s_{0}=1.08$). In particular, the time weighted index $1+2s_{1}-s_{0}$ is determined by the estimate of (\ref{2.53}).
\end{Remark}

\begin{Remark}
Here we give some comments on upper and lower bounds for $s_0,s_1$ in \eqref{s-0s-1}. First, we always assume $s_1<s_0<1+s_1$ in light of the time-weighted energy settings \eqref{2.6}--\eqref{2.7}. In (\ref{2.11}), under $s_0>s_1$, we further require $s_{0}>\frac{3s_{1}^{2}-2s_{1}+2}{3s_{1}}$ so that $\frac{1+2s_{1}-s_{0}}{2-\theta}>1$ with $\theta=\frac{s_{0}-s_{1}}{2+s_{0}-s_{1}}$. And we need $s_1<1$ and $s_0<2s_1$ due to \eqref{2.53}. So one assumes
\begin{align*}
\frac{3s_{1}^{2}-2s_{1}+2}{3s_{1}}<s_{0}<2s_{1},
\end{align*}
which together with \eqref{2.17} and \eqref{2.53} implies
\begin{align*}
\frac{\sqrt{7}-1}{3}<s_{1}<1.
\end{align*}
\end{Remark}

The rest of this paper are organized as follows. In Section \ref{se2}, we first transform system (\ref{1.1}) into a suitable dissipative system, and then we carefully estimate the various energies including the time-weighted part. In Section \ref{se3}, with the help of the previous results, we prove Theorem \ref{th1.1} by a continuous argument. In Appendix \ref{appendix}, we present some useful results which are frequently used in the previous sections.

\section{Energy Framework}\label{se2}
In this section, we first make appropriate transformations to system (\ref{1.1}). Then we give various energies and useful result that play important roles in the subsequent sections.
\subsection{Reformulation}
\ \ \ \\

Since the specific values of the positive coefficients $\mu>0, c>0$ in (\ref{1.1}) are not essential in the present article, so we take $\mu=c=1$ and define $\rho:=\tilde{\rho}-1$ in the rest of this paper. Next, in order to effectively analyze (\ref{1.1}), we provide the definition of the effective tensor $G$ as follows
\begin{equation}\label{2.1}
G:=\tilde{\rho}\mathbb{F}\mathbb{F}^{\top}-\mathbb{I}.
\end{equation}
And we have
\begin{equation}\label{2.2}
(\tilde{\rho}\mathbb{F}\mathbb{F}^{\top})_{t}+(u\cdot\nabla\tilde{\rho})\mathbb{F}\mathbb{F}^{\top}+\tilde{\rho}(u\cdot\nabla\mathbb{F})\mathbb{F}^{\top}
+\tilde{\rho}\mathbb{F}(u\cdot\nabla\mathbb{F}^{\top})=\tilde{\rho}(\nabla u \mathbb{F})\mathbb{F}^{\top}+\tilde{\rho}\mathbb{F}\mathbb{F}^{\top}(\nabla u)^{\top}.
\end{equation}
From (\ref{2.1}) and (\ref{2.2}), we can derive the evolution for effective tensor $G$
\begin{equation}\label{2.3}
G_{t}+u\cdot \nabla G+Q(\nabla u, G)=2D(u),
\end{equation}
where $Q(\nabla u, G)=-\nabla uG-G(\nabla u)^{\top}$ and $D(u)=\frac{1}{2}(\nabla u+(\nabla u)^{\top})$.

Combining (\ref{1.1}) with (\ref{2.3}), the following new system is established
\begin{equation}\label{2.4}
\begin{cases}
u_{t}-\Delta u+u\cdot\nabla u+(1-\frac{1}{\rho+1})\Delta u+\frac{1}{\rho+1}\nabla p+(1-\frac{1}{\rho+1})\diver G=\diver G,\\
G_{t}+ u\cdot\nabla G+Q(\nabla u, G)=2D(u), \\
\diver u=0, \qquad\qquad (x,t)\in\mathbb{R}^{2}\times\mathbb{R}^{+}.
\end{cases}
\end{equation}
\subsection{Various Refined Energies}
\ \ \ \\

Based on the analysis above, for any $t>0$, we can state the following various refined energies for the new system (\ref{2.4}). We first define the following basic energy:
\begin{equation}\label{2.5}
\mathcal{E}_{0}(t):=\sup_{0\leq t'\leq t}(\|u(t')\|_{\dot{H}^{-s_{1}}\cap\dot{H}^{2}}^{2}+\|G(t')\|_{\dot{H}^{-s_{1}}\cap\dot{H}^{2}}^{2})+\int_{0}^{t}\left(\|\nabla u(t')\|_{\dot{H}^{-s_{1}}\cap\dot{H}^{2}}^{2}
+\|\mathbb{P}\diver G(t')\|_{\dot{H}^{-s_{1}}\cap\dot{H}^{1}}^{2}\right)\,dt',
\end{equation}
where $\mathbb{P}=\mathbb{I}+\nabla (-\Delta)^{-1}\diver $ is the Helmholtz projection operator (cf. \cite{Temam1977}).\\
Next, the fractional time-weighted energy given by
\begin{align}\label{2.6}
\mathcal{E}_{1}(t)&:=\sup_{0\leq t'\leq t}(1+t')^{1+2s_{1}-s_{0}}(\|u(t')\|_{\dot{H}^{1+s_{1}-s_{0}}\cap\dot{H}^{2}}^{2}+\|\mathbb{P}\diver  G(t')\|_{\dot{H}^{s_{1}-s_{0}}\cap\dot{H}^{1}}^{2})\nonumber \\
&\ \ \ \ +\int_{0}^{t}(1+t')^{1+2s_{1}-s_{0}}(\|\nabla u(t')\|_{\dot{H}^{1+s_{1}-s_{0}}\cap\dot{H}^{2}}^{2}+\|\mathbb{P}\diver  G(t')\|_{\dot{H}^{1+s_{1}-s_{0}}\cap\dot{H}^{1}}^{2})\,dt',
\end{align}
where $s_{0}, s_{1}$ are arbitrary given positive constants satisfying $s_{1}<s_{0}<s_{1}+1$.\\
Moreover, to obtain the uniform bound of $(\rho, \mathbb{F}-\mathbb{I})$, we also need the following assistant energy
\begin{equation}\label{2.7}
\mathcal{E}_{a}(t):=\sup_{0\leq t'\leq t}\left(\|\rho(t')\|_{H^{2}\cap\dot{H}^{s_{1}-s_{0}}}^{2}+\|(\mathbb{F}-\mathbb{I})(t')\|_{H^{2}}^{2}\right).
\end{equation}

Finally, the total energy $\mathcal{E}_{total}(t)$ is defined as follows
\begin{equation}\label{2.8}
\mathcal{E}_{total}(t):=\mathcal{E}_{0}(t)+\mathcal{E}_{1}(t)+\mathcal{E}_{a}(t).
\end{equation}

Next, we also recall the following useful result.

\begin{Lemma}\label{le2.2}
Let $[\mathbb{F}_{ij}]_{2\times2}, u$ be a some smooth tensor and  two dimensional vector, respectively. Then it holds that
\begin{equation*}
\mathbb{P}\diver (u\cdot\nabla\mathbb{F})=\mathbb{P}(u\cdot\nabla\mathbb{P}\diver \mathbb{F})+\mathbb{P}(\nabla u\cdot\nabla\mathbb{F})-\mathbb{P}(\nabla u\cdot\nabla\Delta^{-1}\diver \diver \mathbb{F}),
\end{equation*}
where the $i$-th components of $\nabla u\cdot\nabla\mathbb{F}$ and $\nabla u\cdot\nabla\Delta^{-1}\diver \diver \mathbb{F}$ are written as
\begin{align*}
&[\nabla u\cdot\nabla\mathbb{F}]_{i}=\sum_{j=0}^{2}\partial_{j}u\cdot\nabla\mathbb{F}_{ij};\\
&[\nabla u\cdot\nabla\Delta^{-1}\diver \diver \mathbb{F}]_{i}=\partial_{i}u\cdot\nabla\Delta^{-1}\diver \diver \mathbb{F}.
\end{align*}
\end{Lemma}
\begin{proof}
See Proposition 3.1 in \cite{Zhu2018}.
\end{proof}

\subsection{Estimates of Assistant Energy $\mathcal{E}_{a}(t)$}

\begin{Lemma}\label{le2.3}
Assume that $\mathcal{E}_{a}(t)$ is defined as in (\ref{2.7}). Then the following estimate is given
\begin{equation}\label{2.9}
\mathcal{E}_{a}(t)\lesssim\mathcal{E}_{a}(0)+\mathcal{E}_{a}(t)\mathcal{E}_{total}^{\frac{1}{2}}(t)
+\mathcal{E}_{a}^{\frac{1}{2}}(t)\mathcal{E}_{total}^{\frac{1}{2}}(t)
\end{equation}
for any $t>0$.
\end{Lemma}

\begin{proof}
Taking $L^{2}-$inner product of the first equation in (\ref{1.1}) with $\rho$, using Lemma \ref{le5.1}, H\"{o}lder and Young's inequalities, we have
\begin{align}\label{2.10}
\frac{1}{2}\frac{d}{dt}\|\rho\|_{L^{2}}^{2}&\leq\|u\cdot\nabla\rho\|_{L^{2}}\|\rho\|_{L^{2}}\nonumber \\
&\lesssim\|u\|_{L^{\infty}}\|\nabla\rho\|_{L^{2}}\|\rho\|_{L^{2}}\nonumber \\
&\lesssim\|\nabla u\|_{\dot{H}^{s_{1}-s_{0}}}^{1-\theta}\|\nabla u\|_{\dot{H}^{2}}^{\theta}\|\nabla\rho\|_{L^{2}}\|\rho\|_{L^{2}},
\end{align}
where $\theta=\frac{s_{0}-s_{1}}{2+s_{0}-s_{1}}$. Integrating (\ref{2.10}) with respect to $t'$ over $(0,t)$,  we get
\begin{align}\label{2.11}
\|\rho(t)\|_{L^{2}}^{2}&\lesssim \|\rho_{0}\|_{L^{2}}^{2}+\mathcal{E}_{a}(t)\int_{0}^{t}\|\nabla u\|_{\dot{H}^{s_{1}-s_{0}}}^{1-\theta}\|\nabla u\|_{\dot{H}^{2}}^{\theta}\,dt' \nonumber \\
&\lesssim \mathcal{E}_{a}(0)+\mathcal{E}_{a}(t)\bigg(\int_{0}^{t}(1+t')^{-\frac{1+2s_{1}-s_{0}}{2}(1-\theta)}(1+t')^{\frac{1+2s_{1}-s_{0}}{2}(1-\theta)}\|\nabla u\|_{\dot{H}^{s_{1}-s_{0}}}^{1-\theta}\nonumber \\
&\ \  \  \ (1+t')^{-\frac{1+2s_{1}-s_{0}}{2}\theta}(1+t')^{\frac{1+2s_{1}-s_{0}}{2}\theta}
\|\nabla u\|_{\dot{H}^{2}}^{\theta}\,dt'\bigg)\nonumber \\
&\lesssim\mathcal{E}_{a}(0)+\mathcal{E}_{a}(t)\sup_{0\leq t'\leq t}\bigg((1+t')^{\frac{1+2s_{1}-s_{0}}{2}}\|\nabla u\|_{\dot{H}^{s_{1}-s_{0}}}\bigg)^{1-\theta}\nonumber\\
&\quad\times\bigg(\int_{0}^{t}
(1+t')^{-\frac{1+2s_{1}-s_{0}}{2}}(1+t')^{\frac{1+2s_{1}-s_{0}}{2}\theta}\|\nabla u\|_{\dot{H}^{2}}^{\theta}\,dt'\bigg)\nonumber \\
&\lesssim\mathcal{E}_{a}(0)+\mathcal{E}_{a}(t)\mathcal{E}_{1}^{\frac{(1-\theta)}{2}}(t)\bigg(\int_{0}^{t}(1+t')^{-\frac{1+2s_{1}-s_{0}}{2-\theta}}\,dt'\bigg)^{\frac{2-\theta}{2}}
\bigg(\int_{0}^{t}(1+t')^{1+2s_{1}-s_{0}}\|\nabla u\|_{\dot{H}^{2}}^{2}\,dt'\bigg)^{\frac{\theta}{2}}\nonumber \\
&\lesssim \mathcal{E}_{a}(0)+\mathcal{E}_{a}(t)\mathcal{E}_{1}^{\frac{1}{2}}(t)\lesssim \mathcal{E}_{a}(0)+\mathcal{E}_{a}(t)\mathcal{E}_{total}^{\frac{1}{2}}(t),
\end{align}
where $s_{0}, s_{1}$ are required to satisfy $\frac{2}{5}<s_{1}<1, \max\{s_{1},\frac{3s_{1}^{2}-2s_{1}+2}{3s_{1}}\}<s_{0}<1+s_{1}$.\\
Applying $\nabla^{2}$ to the first equation of (\ref{1.1}), taking $L^{2}-$inner product of the resulting equations with $\nabla^{2}\rho$ and combining $\diver u=0$, similar to the estimates of (\ref{2.10}) and (\ref{2.11}), we have
\begin{align}\label{2.12}
\|\nabla^{2}\rho(t)\|_{L^{2}}^{2}&\lesssim \|\nabla^{2}\rho_{0}\|_{L^{2}}^{2}+\int_{0}^{t}(\|\nabla u\|_{L^{\infty}}\|\nabla^{2}\rho\|_{L^{2}}^{2}+\|\nabla^{2}u\|_{L^{6}}\|\nabla\rho\|_{L^{3}}\|\nabla^{2}\rho\|_{L^{2}})\,dt'\nonumber \\
&\lesssim \|\nabla^{2}\rho_{0}\|_{L^{2}}^{2}+\int_{0}^{t}(\|\nabla u\|_{\dot{H}^{1+s_{1}-s_{0}}\cap\dot{H}^{2}}\|\nabla^{2}\rho\|_{L^{2}}^{2}+\|\nabla u\|_{\dot{H}^{1+s_{1}-s_{0}}\cap\dot{H}^{2}}\|\rho\|_{H^{2}}^{2})\,dt'\nonumber \\
&\lesssim \|\nabla^{2}\rho_{0}\|_{L^{2}}^{2}+\mathcal{E}_{a}(t)\bigg(\int_{0}^{t}(1+t')^{-\frac{1+2s_{1}-s_{0}}{2}}(1+t')^{\frac{1+2s_{1}-s_{0}}{2}}\|\nabla u\|_{\dot{H}^{1+s_{1}-s_{0}}\cap\dot{H}^{2}}\,dt'\bigg)\nonumber \\
&\lesssim \mathcal{E}_{a}(0)+\mathcal{E}_{a}(t)\mathcal{E}_{1}^{\frac{1}{2}}(t)\lesssim \mathcal{E}_{a}(0)+\mathcal{E}_{a}(t)\mathcal{E}_{total}^{\frac{1}{2}}(t).
\end{align}
Next, taking $L^{2}-$inner product of the third equation in (\ref{1.1}) with $\mathbb{F}-\mathbb{I}$ and integrating the result equation with respect to $t'$ over $(0,t)$, we have
\begin{align}\label{2.13}
\|(\mathbb{F}-\mathbb{I})(t)\|_{L^{2}}^{2}&\lesssim \|\mathbb{F}_{0}-\mathbb{I}\|_{L^{2}}^{2}+\int_{0}^{t}(\|u\cdot\nabla(\mathbb{F}-\mathbb{I}) \|_{L^{2}}+\|\nabla u(\mathbb{F}-\mathbb{I})\|_{L^{2}}+\|\nabla u\|_{L^{2}})\|\mathbb{F}-\mathbb{I}\|_{L^{2}}\,dt' \nonumber \\
&\lesssim \|\mathbb{F}_{0}-\mathbb{I}\|_{L^{2}}^{2}+\int_{0}^{t}[(\|u\|_{L^{\infty}}\|\nabla(\mathbb{F}-\mathbb{I})\|_{L^{2}}+\|\nabla u\|_{L^{\infty}}\|\mathbb{F}-\mathbb{I}\|_{L^{2}})\|\mathbb{F}-\mathbb{I}\|_{L^{2}}+\|\nabla u\|_{L^{2}}\|\mathbb{F}-\mathbb{I}\|_{L^{2}}]\,dt' \nonumber \\
&\lesssim \mathcal{E}_{a}(0)+\mathcal{E}_{a}(t)\int_{0}^{t}(\|u\|_{L^{\infty}}+\|\nabla u\|_{L^{\infty}})\,dt'+\mathcal{E}_{a}^{\frac{1}{2}}(t)\int_{0}^{t}\|\nabla u\|_{L^{2}}\,dt' \nonumber \\
&\lesssim \mathcal{E}_{a}(0)+\mathcal{E}_{a}(t)\int_{0}^{t}(\|\nabla u\|_{\dot{H}^{s_{1}-s_{0}}\cap\dot{H}^{2}}+\|\nabla u\|_{\dot{H}^{1+s_{1}-s_{0}}\cap\dot{H}^{2}})\,dt'+\mathcal{E}_{a}^{\frac{1}{2}}(t)\int_{0}^{t}\|\nabla u\|_{\dot{H}^{s_{1}-s_{0}}\cap\dot{H}^{2}}\,dt' \nonumber \\
&\lesssim \mathcal{E}_{a}(0)+\mathcal{E}_{a}(t)\mathcal{E}_{1}^{\frac{1}{2}}(t)+\mathcal{E}_{a}^{\frac{1}{2}}(t)\mathcal{E}_{1}^{\frac{1}{2}}(t)\lesssim \mathcal{E}_{a}(0)+\mathcal{E}_{a}(t)\mathcal{E}_{total}^{\frac{1}{2}}(t)+\mathcal{E}_{a}^{\frac{1}{2}}(t)\mathcal{E}_{total}^{\frac{1}{2}}(t).
\end{align}
Applying $\nabla^{2}$ to the third equation of (\ref{1.1}), taking $L^{2}-$inner product of the result equation with $\nabla^{2}(\mathbb{F}-\mathbb{I})$, and integrating $t'$ over $(0,t)$, we have
\begin{align}\label{2.14}
\|\nabla^{2}(\mathbb{F}-\mathbb{I})(t)\|_{L^{2}}^{2}&\lesssim \|\nabla^{2}(\mathbb{F}_{0}-\mathbb{I})\|_{L^{2}}^{2}+\int_{0}^{t}(\|\nabla u\|_{L^{\infty}}\|\nabla^{2}(\mathbb{F}-\mathbb{I})\|_{L^{2}}^{2}+\|\nabla^{2} u\|_{L^{6}}\|\nabla(\mathbb{F}-\mathbb{I})\|_{L^{3}}\|\nabla^{2}(\mathbb{F}-\mathbb{I})\|_{L^{2}}\nonumber \\
&\ \ \ \ +\|\nabla^{3} u\|_{L^{2}}\|\mathbb{F}-\mathbb{I}\|_{L^{\infty}}\|\nabla^{2}(\mathbb{F}-\mathbb{I})\|_{L^{2}}+\|\nabla^{3} u\|_{L^{2}}\|\nabla^{2}(\mathbb{F}-\mathbb{I})\|_{L^{2}})\,dt' \nonumber \\
&\lesssim \mathcal{E}_{a}(0)+\mathcal{E}_{a}(t)\mathcal{E}_{1}^{\frac{1}{2}}(t)+\mathcal{E}_{a}^{\frac{1}{2}}(t)\mathcal{E}_{1}^{\frac{1}{2}}(t) \nonumber \\
&\lesssim \mathcal{E}_{a}(0)+\mathcal{E}_{a}(t)\mathcal{E}_{total}^{\frac{1}{2}}(t)+\mathcal{E}_{a}^{\frac{1}{2}}(t)\mathcal{E}_{total}^{\frac{1}{2}}(t).
\end{align}

Next, we consider the case of $0<\alpha<2$. Applying Lemma \ref{le5.1} and Young's inequality, we have
\begin{align}\label{2.15}
\||\nabla|^{\alpha}(\rho,(\mathbb{F}-\mathbb{I}))(t)\|_{L^{2}}^{2}&\lesssim[\|(\rho,(\mathbb{F}-\mathbb{I}))(t)\|_{L^{2}}^{\frac{2-\alpha}{2}}\|\nabla^{2}(\rho,(\mathbb{F}-\mathbb{I}))(t)
\|_{L^{2}}^{\frac{\alpha}{2}}]^{2}\nonumber \\
&\lesssim\|(\rho,(\mathbb{F}-\mathbb{I}))(t)\|_{L^{2}}^{2}+\|\nabla^{2}(\rho,(\mathbb{F}-\mathbb{I}))(t)\|_{L^{2}}^{2}\nonumber \\
&\lesssim\mathcal{E}_{a}(0)+\mathcal{E}_{a}(t)\mathcal{E}_{total}^{\frac{1}{2}}(t)+\mathcal{E}_{a}^{\frac{1}{2}}(t)\mathcal{E}_{total}^{\frac{1}{2}}(t).
\end{align}
Finally, let's establish the following estimate for $\|\rho(t')\|_{\dot{H}^{s_{1}-s_{0}}}$:
\begin{align}\label{2.16}
\frac{1}{2}\frac{d}{dt}\|\rho\|_{\dot{H}^{s_{1}-s_{0}}}^{2}&=-\left\langle u\cdot\nabla\rho,\rho\right\rangle_{\dot{H}^{s_{1}-s_{0}}}\nonumber \\
&\leq\|u\cdot\nabla\rho\|_{\dot{H}^{s_{1}-s_{0}}}\|\rho\|_{\dot{H}^{s_{1}-s_{0}}}\nonumber \\
&=\|\diver(\rho u)\|_{\dot{H}^{s_{1}-s_{0}}}\|\rho\|_{\dot{H}^{s_{1}-s_{0}}}\nonumber \\
&\lesssim\big(\|\vert\nabla\vert^{1+s_{1}-s_{0}}\rho\|_{L^{2}}\|u\|_{L^{\infty}}
+\|\rho\|_{L^{2}}\|\vert\nabla\vert^{1+s_{1}-s_{0}}u\|_{L^{\infty}}\big)\|\rho\|_{\dot{H}^{s_{1}-s_{0}}}\nonumber \\
&\lesssim\big(\|\vert\nabla\vert^{1+s_{1}-s_{0}}\rho\|_{L^{2}}\|\nabla u\|_{\dot{H}^{s_{1}-s_{0}}}^{1-\theta_{1}}\|\nabla u\|_{\dot{H}^{2}}^{\theta_{1}}+\|\rho\|_{L^{2}}\|\nabla u\|_{\dot{H}^{s_{1}-s_{0}}}^{1-\theta_{2}}\|\nabla u\|_{\dot{H}^{2}}^{\theta_{2}}\big)\|\rho\|_{\dot{H}^{s_{1}-s_{0}}},
\end{align}
where $\theta_{1}=\frac{s_{0}-s_{1}}{2+s_{0}-s_{1}}, \theta_{2}=\frac{1}{2+s_{0}-s_{1}}$.\\
Integrating (\ref{2.16}) from $0$ to $t$, we have
\begin{align}\label{2.17}
\|\rho(t)\|_{\dot{H}^{s_{1}-s_{0}}}^{2}&\lesssim \|\rho_{0}\|_{\dot{H}^{s_{1}-s_{0}}}^{2}+\mathcal{E}_{a}(t)\int_{0}^{t}\big(\|\nabla u\|_{\dot{H}^{s_{1}-s_{0}}}^{1-\theta_{1}}\|\nabla u\|_{\dot{H}^{2}}^{\theta_{1}}+\|\nabla u\|_{\dot{H}^{s_{1}-s_{0}}}^{1-\theta_{2}}\|\nabla u\|_{\dot{H}^{2}}^{\theta_{2}}\big)\,dt' \nonumber \\
&\lesssim\mathcal{E}_{a}(0)+\mathcal{E}_{a}(t)\mathcal{E}_{1}^{\frac{1}{2}}(t)+\mathcal{E}_{a}(t)\sup_{0\leq t'\leq t}\bigg((1+t')^{\frac{1+2s_{1}-s_{0}}{2}}\|\nabla u\|_{\dot{H}^{s_{1}-s_{0}}}\bigg)^{1-\theta_{2}}\nonumber \\
&\ \ \ \  \times\bigg(\int_{0}^{t}(1+t')^{-\frac{1+2s_{1}-s_{0}}{2}}(1+t')^{\frac{1+2s_{1}-s_{0}}{2}\theta_{2}}\|\nabla u\|_{\dot{H}^{2}}^{\theta_{2}}\,dt'\bigg)\nonumber \\
&\lesssim\mathcal{E}_{a}(0)+\mathcal{E}_{a}(t)\mathcal{E}_{1}^{\frac{1}{2}}(t)+\mathcal{E}_{a}(t)
\mathcal{E}_{1}^{\frac{1-\theta_{2}}{2}}(t)\bigg(\int_{0}^{t}(1+t')^{-\frac{1+2s_{1}-s_{0}}{2-\theta_{2}}}\,dt'\bigg)^{\frac{2-\theta_{2}}{2}}
\nonumber \\
&\ \ \ \  \times\bigg(\int_{0}^{t}(1+t')^{1+2s_{1}-s_{0}}\|\nabla u\|_{\dot{H}^{2}}^{2}\,dt'\bigg)^{\frac{\theta_{2}}{2}}\nonumber \\
&\lesssim \mathcal{E}_{a}(0)+\mathcal{E}_{a}(t)\mathcal{E}_{1}^{\frac{1}{2}}(t)\lesssim \mathcal{E}_{a}(0)+\mathcal{E}_{a}(t)\mathcal{E}_{total}^{\frac{1}{2}}(t),
\end{align}
where $s_{0}, s_{1}$ are required to satisfy $\frac{2}{5}<s_{1}<1, \max\{s_{1},\frac{3s_{1}^{2}-2s_{1}+2}{3s_{1}}\}<s_{0}<1+s_{1}$.\\
Combining the results of (\ref{2.11})-(\ref{2.15}) and (\ref{2.17}), we immediately get (\ref{2.9}). This completes the proof of Lemma \ref{le2.3}.

\end{proof}

\subsection{Estimates of Basic Energy $\mathcal{E}_{0}(t)$}
\ \ \ \\
In this subsection, in order to estimate the basic energy $\mathcal{E}_{0}(t)$, the main difficulty is to deal with the pressure term $\frac{1}{\rho+1}\nabla p$ in the first equation of system (\ref{2.4}). To this end, we first transform the pressure $p$ to $\tilde{p}$ by
\begin{equation}\label{2.18}
\nabla\tilde{p}:=\frac{1}{\rho+1}\nabla p.
\end{equation}
\textbf{Remark 2.1.}
The definition in (\ref{2.18}) is valid. In fact, it can be obtained from \cite{Galdi2011} that $(u,p)$ is a solution of the following Stokes problem
\begin{align}\label{2.19}
\begin{cases}
-\Delta u+\nabla p=-(\rho+1)u_{t}-(\rho+1)u\cdot\nabla u+\diver G, & x\in\mathbb{R}^{2},\\
\diver u=0, & x\in\mathbb{R}^{2}.
\end{cases}
\end{align}
Applying $-\diver $ to (\ref{2.18}), we have
\begin{equation}\label{2.018}
-\Delta\tilde{p}=-\diver\Big(\frac{1}{\rho+1}\nabla p\Big).
\end{equation}
By solving the above Poisson equation for $\tilde{p}$, we can get
\begin{equation}\label{2.019}
\tilde{p}=\int_{\mathbb{R}^{2}}\Phi(x-y)\diver\Big(\frac{1}{\rho(y)+1}\nabla p(y)\Big)\,dy,
\end{equation}
where $\Phi(x), 0\neq x\in\mathbb{R}^{2}$, is the fundamental solution of the Laplacian equation (cf. \cite{Evans1998}). Next, using Lemma \ref{le5.5}, we have the following estimate
\begin{align*}
\|\nabla\tilde{p}\|_{L^{2}}&\lesssim \|\frac{1}{1+\rho}\|_{L^{\infty}}\|\nabla p\|_{L^{2}}  \\
&\lesssim\|\nabla p\|_{L^{2}}\lesssim\|f\|_{L^{2}},
\end{align*}
where $f=-(\rho+1)u_{t}-(\rho+1)u\cdot\nabla u+\diver G$.\\
Considering (\ref{2.18}), the system $(\ref{2.4})$ can be written as
\begin{equation}\label{2.20}
\begin{cases}
u_{t}-\Delta u+u\cdot\nabla u+(1-\frac{1}{\rho+1})\Delta u+\nabla\tilde{p}+(1-\frac{1}{\rho+1})\diver G=\diver G,\\
G_{t}+ u\cdot\nabla G+Q(\nabla u, G)=2D(u), \\
\diver u=0, \qquad\qquad (x,t)\in\mathbb{R}^{2}\times\mathbb{R}^{+}.
\end{cases}
\end{equation}

The estimate of the basic energy $\mathcal{E}_{0}(t)$ can be established in the following lemma.

\begin{Lemma}\label{le2.4}
Assume that $\mathcal{E}_{0}(t)$ is defined as in (\ref{2.5}). Then the following estimate is given
\begin{equation}\label{2.21}
\mathcal{E}_{0}(t)\lesssim\mathcal{E}_{01}(0)+\mathcal{E}_{total}(t)+\mathcal{E}_{total}^{\frac{3}{2}}(t)+\mathcal{E}_{total}^{\frac{9}{4}}(t),
\end{equation}
for any $t>0$.
\end{Lemma}

\begin{proof}
We divide the estimate of the basic energy $\mathcal{E}_{0}(t)$ into two steps. To achieve this, we first divide the basic energy $\mathcal{E}_{0}(t)$ into the following two parts
\begin{align*}
&\mathcal{E}_{01}(t)=\sup_{0\leq t'\leq t}(\|u(t')\|_{\dot{H}^{-s_{1}}\cap\dot{H}^{2}}^{2}+\|G(t')\|_{\dot{H}^{-s_{1}}\cap\dot{H}^{2}}^{2})+\int_{0}^{t}\|\nabla u(t')\|_{\dot{H}^{-s_{1}}\cap\dot{H}^{2}}^{2}\,dt', \\
&\mathcal{E}_{02}(t)=\int_{0}^{t}\|\mathbb{P}\diver G(t')\|_{\dot{H}^{-s_{1}}\cap\dot{H}^{1}}^{2}\,dt'.
\end{align*}
\textbf{Step 1.} The estimate of $\mathcal{E}_{01}(t)$.\\
Applying the operator $\nabla^{k}(k=-s_{1},2)$ to (\ref{2.20}). Then taking inner product of $(2\nabla^{k}u,\nabla^{k}G)$ for the corresponding equations in $(\ref{2.20})$. We get
\begin{align}\label{2.22}
\frac{1}{2}\frac{d}{dt}\left(2\|u\|_{\dot{H}^{-s_{1}}\cap\dot{H}^{2}}^{2}+\|G\|_{\dot{H}^{-s_{1}}\cap\dot{H}^{2}}^{2}\right)+\|\nabla u\|_{\dot{H}^{-s_{1}}\cap\dot{H}^{2}}^{2}=I_{1}+I_{2}+I_{3}+I_{4},
\end{align}
where
\begin{align*}
&I_{1}=\left\langle\diver  G,2u\right\rangle_{\dot{H}^{-s_{1}}\cap\dot{H}^{2}}+\left\langle2D(u),G\right\rangle_{\dot{H}^{-s_{1}}\cap\dot{H}^{2}};\\
&I_{2}=-\left\langle u\cdot\nabla u,2u\right\rangle_{\dot{H}^{-s_{1}}\cap\dot{H}^{2}}-\left\langle u\cdot\nabla G,G\right\rangle_{\dot{H}^{-s_{1}}\cap\dot{H}^{2}};\\
&I_{3}=-\left\langle Q(\nabla u,G),G\right\rangle_{\dot{H}^{-s_{1}}\cap\dot{H}^{2}};\\
&I_{4}=-2\left\langle\frac{\rho}{\rho+1}(\Delta u+\diver G),u\right\rangle_{\dot{H}^{-s_{1}}\cap\dot{H}^{2}}.
\end{align*}
For $I_{1}$, it follows from integration by parts and the symmetry $G_{ij}=G_{ji}$ that
\begin{align}\label{2.23}
I_{1}&=\sum_{k=-s_{1},2}\int_{\mathbb{R}^{2}}2\bigg(\nabla^{k}\diver  G\nabla^{k}u+\sum_{i,j=1}^{2}\nabla^{k}\frac{\partial_{j}u_{i}+\partial_{i}u_{j}}{2}\nabla^{k}G_{ij}\bigg)\,dx\nonumber \\
&=\sum_{k=-s_{1},2}2\int_{\mathbb{R}^{2}}\bigg(\nabla^{k}\diver  G\nabla^{k}u-\sum_{i,j=1}^{2}\frac{\nabla^{k}u_{i}\nabla^{k}\partial_{j}G_{ij}+\nabla^{k}u_{j}\nabla^{k}\partial_{i}G_{ij}}{2}\bigg)\,dx\nonumber \\
&=\sum_{k=-s_{1},2}2\int_{\mathbb{R}^{2}}\bigg(\nabla^{k}\diver  G\nabla^{k}u-\sum_{i,j=1}^{2}\nabla^{k}u_{i}\nabla^{k}\partial_{j}G_{ij}\bigg)\,dx\nonumber \\
&=0.
\end{align}
For $I_{2}$, using the divergence-free condition, H\"{o}lder's inequality, Lemma \ref{le5.3} and Sobolev imbedding theorem, we have
\begin{align}\label{2.24}
I_{2}&=-2\left\langle[|\nabla|^{-s_{1}},u\cdot\nabla] u,|\nabla|^{-s_{1}}u\right\rangle_{L^{2}}-\left\langle[|\nabla|^{-s_{1}},u\cdot\nabla] G,|\nabla|^{-s_{1}}G\right\rangle_{L^{2}}\nonumber \\
&\ \ \ \ -2\left\langle\nabla^{2}(u\cdot\nabla u),\nabla^{2}u\right\rangle_{L^{2}}-\left\langle\nabla^{2}(u\cdot\nabla G),\nabla^{2}G\right\rangle_{L^{2}}\nonumber \\
&\lesssim\|[|\nabla|^{-s_{1}},u\cdot\nabla] u\|_{L^{2}}\|u\|_{\dot{H}^{-s_{1}}}+\|[|\nabla|^{-s_{1}},u\cdot\nabla] G\|_{L^{2}}\|G\|_{\dot{H}^{-s_{1}}}+\|\nabla u\|_{L^{\infty}}\|\nabla^{2} u\|_{L^{2}}^{2} \nonumber \\
&\ \ \ \ +\|\nabla^{2} u\|_{L^{3}}\|\nabla G\|_{L^{6}}\|\nabla^{2} G\|_{L^{2}}+\|\nabla u\|_{L^{\infty}}\|\nabla^{2} G\|_{L^{2}}^{2}\nonumber \\
&\lesssim\|\nabla^{2} u\|_{\dot{H}^{-s_{0}+s_{1}}\cap\dot{H}^{s_{0}-s_{1}}}\|(u,G)\|_{\dot{H}^{-s_{1}}}^{2}+\|\nabla^{2} u\|_{\dot{H}^{-s_{0}+s_{1}}\cap\dot{H}^{1}}\|(u,G)\|_{\dot{H}^{2}}^{2}\nonumber \\
&\ \ \ \ +\|\nabla u\|_{\dot{H}^{-s_{0}+s_{1}}\cap\dot{H}^{2}}\|G\|_{\dot{H}^{1-s_{1}}\cap\dot{H}^{2}}\|\nabla^{2} G\|_{L^{2}}.
\end{align}
Integrating (\ref{2.24}) with respect to time yields
\begin{align}\label{2.25}
\int_{0}^{t}|I_{2}(t')|\,dt'&\lesssim\mathcal{E}_{0}(t)\mathcal{E}_{1}^{\frac{1}{2}}(t)+\mathcal{E}_{0}^{\frac{1}{2}}(t)\int_{0}^{t}
\bigg(\|u\|_{\dot{H}^{1+s_{1}-s_{0}}}\|\nabla G\|_{\dot{H}^{-s_{1}}}+\|u\|_{\dot{H}^{1+s_{1}-s_{0}}}\|\nabla^{2} G\|_{L^{2}}\bigg)\nonumber \\
&\ \ \ \ +\|\nabla u\|_{\dot{H}^{2}}\|\nabla G\|_{\dot{H}^{-s_{1}}}+\|\nabla u\|_{\dot{H}^{2}}\|\nabla^{2} G\|_{L^{2}}\bigg)\,dt'\nonumber \\
&\lesssim\mathcal{E}_{0}(t)\mathcal{E}_{1}^{\frac{1}{2}}(t)+\mathcal{E}_{0}^{\frac{3}{2}}(t)\lesssim\mathcal{E}_{0}^{\frac{3}{2}}(t)
+\mathcal{E}_{1}^{\frac{3}{2}}(t).
\end{align}
For $I_{3}$, by Lemma \ref{le5.1}, Lemma \ref{le5.4}, and H\"{o}lder's inequality, it holds that
\begin{align}\label{2.26}
|I_{3}|&\lesssim\|Q(\nabla u,G)\|_{\dot{H}^{-s_{1}}}\|G\|_{\dot{H}^{-s_{1}}}+\|Q(\nabla u,G)\|_{\dot{H}^{2}}\|G\|_{\dot{H}^{2}}\nonumber \\
&\lesssim\|\nabla^{2} u\|_{\dot{H}^{-s_{0}+s_{1}}\cap\dot{H}^{s_{0}-s_{1}}}\|G\|_{\dot{H}^{-s_{1}}}^{2}+(\|G\|_{L^{\infty}}\|\nabla^{3} u\|_{L^{2}}+\|\nabla^{2} u\|_{L^{3}}\|\nabla G\|_{L^{6}}+\|\nabla u\|_{L^{\infty}}\|\nabla^{2} G\|_{L^{2}})\|G\|_{\dot{H}^{2}}\nonumber \\
&\lesssim\|\nabla^{2} u\|_{\dot{H}^{-s_{0}+s_{1}}\cap\dot{H}^{s_{0}-s_{1}}}\|G\|_{\dot{H}^{-s_{1}}}^{2}+(\|G\|_{\dot{H}^{1-s_{1}}\cap\dot{H}^{2}}\|\nabla u\|_{\dot{H}^{2}}+\|\nabla u\|_{\dot{H}^{-s_{0}+s_{1}}\cap\dot{H}^{2}}\|G\|_{\dot{H}^{1-s_{1}}\cap\dot{H}^{2}})\|G\|_{\dot{H}^{2}}\nonumber \\
&\ \ \ \ +\|\nabla u\|_{\dot{H}^{1+s_{1}-s_{0}}\cap\dot{H}^{2}}\|\nabla^{2} G\|_{L^{2}}\|G\|_{\dot{H}^{2}}.
\end{align}
Integrating (\ref{2.26}) from $0$ to $t$, we have
\begin{equation}\label{2.27}
\int_{0}^{t}|I_{3}(t')|\,dt'\lesssim\mathcal{E}_{0}(t)\mathcal{E}_{1}^{\frac{1}{2}}(t)+\mathcal{E}_{0}^{\frac{3}{2}}(t)
\lesssim\mathcal{E}_{0}^{\frac{3}{2}}(t)+\mathcal{E}_{1}^{\frac{3}{2}}(t).
\end{equation}
Let
\begin{align*}
g(\rho):=\frac{\rho}{\rho+1}.
\end{align*}
For $I_{4}$, applying H\"{o}lder's inequality, Lemmas \ref{le5.1}--\ref{le5.2} and Lemma \ref{le5.4}, we have
\begin{align}\label{2.28}
|I_{4}|&\lesssim\|g(\rho)(\diver  G+\Delta u)\|_{\dot{H}^{-s_{1}}}\|u\|_{\dot{H}^{-s_{1}}}+\|g(\rho)(\diver  G+\Delta u)\|_{\dot{H}^{1}} \|u\|_{\dot{H}^{3}}\nonumber \\
&\lesssim\Big(\|g(\rho)\diver  G\|_{\dot{H}^{-s_{1}}}+\|g(\rho)\Delta u\|_{\dot{H}^{-s_{1}}}\Big)\|u\|_{\dot{H}^{-s_{1}}}+\Big(\|g(\rho)\|_{L^{\infty}}\|\diver  G+\Delta u\|_{\dot{H}^{1}}\nonumber \\
&\ \ \ \ +\|\nabla g(\rho)\|_{L^{3}}\|\diver  G+\Delta u\|_{L^{6}}\Big)\|u\|_{\dot{H}^{3}}\nonumber \\
&\lesssim\Big(\|\diver \big(g(\rho)G\big)-\nabla(g(\rho))G\|_{\dot{H}^{-s_{1}}}+\|\nabla\rho\|_{\dot{H}^{-s_{0}+s_{1}}\cap\dot{H}^{s_{0}-s_{1}}}\|\Delta u\|_{\dot{H}^{-s_{1}}}\Big)\|u\|_{\dot{H}^{-s_{1}}}+\Big(\|g(\rho)\|_{L^{\infty}}\|\diver  G+\Delta u\|_{\dot{H}^{1}}\nonumber \\
&\ \ \ \ +\|\nabla g(\rho)\|_{L^{3}}\|\diver  G+\Delta u\|_{L^{6}}\Big)\|u\|_{\dot{H}^{3}}\nonumber \\
&\lesssim\Big(\|\diver \big(g(\rho)G\big)\|_{\dot{H}^{-s_{1}}}+\|\nabla\big(g(\rho)\big)G\|_{\dot{H}^{-s_{1}}}+\|\rho\|_{H^{2}}\|\nabla u\|_{\dot{H}^{1+s_{1}-s_{0}}\cap\dot{H}^{2}}\Big)\|u\|_{\dot{H}^{-s_{1}}}+\Big(\|\nabla^{2}G\|_{L^{2}}+\|\nabla u\|_{\dot{H}^{2}}\nonumber \\
&\ \ \ \ +\|G\|_{\dot{H}^{1-s_{1}}\cap\dot{H}^{2}}+\|\nabla u\|_{\dot{H}^{-s_{1}}\cap\dot{H}^{2}}\Big)\|\rho\|_{H^{2}}\|u\|_{\dot{H}^{3}}\nonumber \\
&\lesssim\Big(\|\rho\|_{L^{3}}\||\nabla|^{1-s_{1}} G\|_{L^{6}}+\||\nabla|^{1-s_{1}} \rho\|_{L^{2}}\|G\|_{L^{\infty}}+\|\nabla G\|_{\dot{H}^{s_{0}-(1+s_{1})}\cap\dot{H}^{1+s_{1}-s_{0}}}\|\rho\|_{\dot{H}^{1-s_{1}}}\nonumber \\
&\ \ \ \ +\|\rho\|_{H^{2}}\|\nabla u\|_{\dot{H}^{1+s_{1}-s_{0}}\cap\dot{H}^{2}}\Big)\|u\|_{\dot{H}^{-s_{1}}}+\Big(\|\nabla^{2}G\|_{L^{2}}+\|\nabla u\|_{\dot{H}^{2}}+\|G\|_{\dot{H}^{1-s_{1}}\cap\dot{H}^{2}}+\|\nabla u\|_{\dot{H}^{-s_{1}}\cap\dot{H}^{2}}\Big)\|\rho\|_{H^{2}}\|u\|_{\dot{H}^{3}}.
\end{align}
Integrating (\ref{2.28}) from $0$ to $t$, we have
\begin{align}\label{2.29}
\int_{0}^{t}|I_{4}(t')|\,dt'&\lesssim\mathcal{E}_{a}^{\frac{1}{2}}(t)\mathcal{E}_{0}^{\frac{1}{2}}(t)\mathcal{E}_{1}^{\frac{1}{2}}(t)
+\mathcal{E}_{a}^{\frac{1}{2}}(t)\mathcal{E}_{0}(t)\nonumber \\
&\lesssim\mathcal{E}_{a}^{\frac{3}{2}}(t)+\mathcal{E}_{0}^{\frac{3}{2}}(t)+\mathcal{E}_{1}^{\frac{3}{2}}(t).
\end{align}
Finally, putting (\ref{2.23}),(\ref{2.25}), (\ref{2.27}) and (\ref{2.29}) into (\ref{2.22}),  we get
\begin{align}\label{2.30}
\mathcal{E}_{01}(t)&\lesssim\|(u_{0},G_{0})\|_{\dot{H}^{-s_{1}}\cap\dot{H}^{2}}^{2}+\mathcal{E}_{a}^{\frac{3}{2}}(t)+\mathcal{E}_{0}^{\frac{3}{2}}(t)+\mathcal{E}_{1}^{\frac{3}{2}}(t)\nonumber \\
&\lesssim\mathcal{E}_{0}(0)+\mathcal{E}_{total}^{\frac{3}{2}}(t).
\end{align}

\textbf{Step 2.} The estimate of $\mathcal{E}_{02}(t)$.\\
Multiplying the first equation of system (\ref{2.4}) by $\rho+1$,  and using operator $\mathbb{P}$ on the resulting equations, we have
\begin{equation}\label{2.31}
\mathbb{P}(\rho u_{t})+u_{t}-\Delta u+\mathbb{P}(\rho u\cdot\nabla u)+\mathbb{P}(u\cdot\nabla u)=\mathbb{P}(\diver  G).
\end{equation}
Applying $\nabla^{k}(k=-s_{1},1)$ to (\ref{2.31}), and taking inner product with $\nabla^{k}\mathbb{P}(\diver  G)$, one gets
\begin{equation}\label{2.32}
\|\mathbb{P}(\diver  G)\|_{\dot{H}^{-s_{1}}\cap\dot{H}^{1}}^{2}=I_{5}+I_{6}+I_{7}+I_{8}+I_{9},
\end{equation}
where
\begin{align*}
&I_{5}=-\left\langle\Delta u,\mathbb{P}(\diver  G)\right\rangle_{\dot{H}^{-s_{1}}\cap\dot{H}^{1}};\\
&I_{6}=\left\langle\mathbb{P}(u\cdot\nabla u),\mathbb{P}(\diver  G)\right\rangle_{\dot{H}^{-s_{1}}\cap\dot{H}^{1}};\\
&I_{7}=\left\langle\mathbb{P}(\rho u\cdot\nabla u),\mathbb{P}(\diver  G)\right\rangle_{\dot{H}^{-s_{1}}\cap\dot{H}^{1}};\\
&I_{8}=\left\langle u_{t},\mathbb{P}(\diver  G)\right\rangle_{\dot{H}^{-s_{1}}\cap\dot{H}^{1}};\\
&I_{9}=\left\langle\mathbb{P}(\rho u_{t}),\mathbb{P}(\diver  G)\right\rangle_{\dot{H}^{-s_{1}}\cap\dot{H}^{1}}.
\end{align*}
For $I_{5}$, it is easy to see that
\begin{align}\label{2.33}
\int_{0}^{t}|I_{5}(t')|\,dt'&\lesssim\int_{0}^{t}\|\Delta u\|_{\dot{H}^{-s_{1}}\cap\dot{H}^{1}}\|\mathbb{P}(\diver  G)\|_{\dot{H}^{-s_{1}}\cap\dot{H}^{1}}\,dt'\nonumber \\
&\lesssim\mathcal{E}_{01}^{\frac{1}{2}}(t)\mathcal{E}_{02}^{\frac{1}{2}}(t).
\end{align}
For $I_{6}$, due to the boundedness of Riesz operator $\mathcal{R}_{i}=(-\Delta)^{-\frac{1}{2}}\nabla_{i}$ from $L^{2}$ into itself, it holds for any vector $u$, $\|\mathbb{P}u\|_{L^{2}}\lesssim\|u\|_{L^{2}}$. Using Lemma \ref{le5.2} and Lemma \ref{le5.4}, we may conclude that
\begin{align}\label{2.34}
I_{6}&\lesssim\|u\cdot\nabla u\|_{\dot{H}^{-s_{1}}}\|\mathbb{P}(\diver  G)\|_{\dot{H}^{-s_{1}}}+\|u\cdot\nabla u\|_{\dot{H}^{1}}\|\mathbb{P}(\diver  G)\|_{\dot{H}^{1}}\nonumber \\
&\lesssim\|\nabla^{2} u\|_{\dot{H}^{-s_{0}+s_{1}}\cap\dot{H}^{s_{0}-s_{1}}}\|u\|_{\dot{H}^{-s_{1}}}\|\mathbb{P}(\diver  G)\|_{\dot{H}^{-s_{1}}}+\Big(\|u\|_{L^{\infty}}\|\nabla^{2}u\|_{L^{2}}\nonumber \\
&\ \  \  \ +\|\nabla u\|_{L^{6}}\|\nabla u\|_{L^{3}}\Big)\|\mathbb{P}(\diver  G)\|_{\dot{H}^{1}}\nonumber \\
&\lesssim\|\nabla u\|_{\dot{H}^{-s_{1}}\cap\dot{H}^{2}}\|u\|_{\dot{H}^{-s_{1}}}\|\mathbb{P}(\diver  G)\|_{\dot{H}^{-s_{1}}}+\Big(\|u\|_{\dot{H}^{1-s_{1}}\cap\dot{H}^{2}}\|\nabla^{2}u\|_{L^{2}}\nonumber \\
&\ \  \  \ +\|u\|_{\dot{H}^{1+s_{1}-s_{0}}\cap\dot{H}^{2}}^{2}\Big)\|\mathbb{P}(\diver  G)\|_{\dot{H}^{1}}.
\end{align}
Integrating (\ref{2.34}) from $0$ to $t$, we arrive at
\begin{align}\label{2.35}
\int_{0}^{t}|I_{6}(t')|\,dt'&\lesssim\sup_{0\leq t'\leq t}\|u\|_{\dot{H}^{-s_{1}}}\int_{0}^{t}\|\nabla u\|_{\dot{H}^{-s_{1}}\cap\dot{H}^{2}}\|\mathbb{P}(\diver  G)\|_{\dot{H}^{-s_{1}}}\,dt'\nonumber\\
&\ \ \ \ +\sup_{0\leq t'\leq t}\|u\|_{\dot{H}^{2}}\int_{0}^{t}\|\nabla u\|_{\dot{H}^{-s_{1}}}\|\mathbb{P}(\diver  G)\|_{\dot{H}^{1}}\,dt'\nonumber \\
&\ \  \  \ +\sup_{0\leq t'\leq t}(1+t')^{1+2s_{1}-s_{0}}\|u\|_{\dot{H}^{1+s_{1}-s_{0}}\cap\dot{H}^{2}}^{2}\int_{0}^{t}(1+t')^{-(1+2s_{1}-s_{0})}\|\mathbb{P}(\diver  G)\|_{\dot{H}^{1}}\,dt'\nonumber \\
&\lesssim\mathcal{E}_{0}^{\frac{3}{2}}(t)+\mathcal{E}_{1}(t)\mathcal{E}_{0}^{\frac{1}{2}}(t)\lesssim\mathcal{E}_{0}^{\frac{3}{2}}(t)
+\mathcal{E}_{1}^{\frac{3}{2}}(t).
\end{align}
For $I_{7}$, using H\"{o}lder's inequality, Lemmas \ref{le5.1}--\ref{le5.2} and Lemma \ref{le5.4}, one gets
\begin{align}\label{2.36}
I_{7}&\lesssim\|\rho u\cdot\nabla u\|_{\dot{H}^{-s_{1}}}\|\mathbb{P}(\diver  G)\|_{\dot{H}^{-s_{1}}}+\|\rho u\cdot\nabla u\|_{\dot{H}^{1}}\|\mathbb{P}(\diver  G)\|_{\dot{H}^{1}}\nonumber \\
&\lesssim\|\nabla(\rho u)\|_{\dot{H}^{-s_{1}}\cap\dot{H}^{s_{1}}}\|\nabla u\|_{\dot{H}^{-s_{1}}}\|\mathbb{P}(\diver  G)\|_{\dot{H}^{-s_{1}}}+\Big(\|\rho\|_{L^{\infty}}\|u\cdot\nabla u\|_{\dot{H}^{1}}\nonumber \\
&\ \  \  \ +\|\nabla \rho\|_{L^{3}}\|u\cdot\nabla u\|_{L^{6}}\Big)\|\mathbb{P}(\diver  G)\|_{\dot{H}^{1}}\nonumber \\
&\lesssim\Big[(\||\nabla|^{1-s_{1}} \rho\|_{L^{2}}+\||\nabla|^{1+s_{1}} \rho\|_{L^{2}})\|u\|_{L^{\infty}}+\|\rho\|_{L^{3}}\||\nabla|^{1-s_{1}} u\|_{L^{6}}+\|\rho\|_{L^{\infty}}\||\nabla|^{1+s_{1}} u\|_{L^{2}}\Big]\|\nabla u\|_{\dot{H}^{-s_{1}}}\|\mathbb{P}(\diver  G)\|_{\dot{H}^{-s_{1}}}\nonumber \\
&\ \  \  \ +\Big(\|u\cdot\nabla u\|_{\dot{H}^{1}}+\|u\cdot\nabla u\|_{L^{6}}\Big)\|\rho\|_{H^{2}}\|\mathbb{P}(\diver  G)\|_{\dot{H}^{1}}\nonumber \\
&\lesssim\|u\|_{\dot{H}^{1+s_{1}-s_{0}}\cap\dot{H}^{2}}\|\rho\|_{H^{2}}\|\nabla u\|_{\dot{H}^{-s_{1}}}\|\mathbb{P}(\diver  G)\|_{\dot{H}^{-s_{1}}}+\|u\|_{\dot{H}^{1+s_{1}-s_{0}}\cap\dot{H}^{2}}^{2}\|\rho\|_{H^{2}}\|\mathbb{P}(\diver  G)\|_{\dot{H}^{1}}.
\end{align}
Integrating (\ref{2.36}) from $0$ to $t$, using H\"{o}lder's inequality, we get
\begin{equation}\label{2.37}
\int_{0}^{t}|I_{7}(t')|\,dt'\lesssim\mathcal{E}_{0}(t)\mathcal{E}_{a}^{\frac{1}{2}}(t)\mathcal{E}_{1}^{\frac{1}{2}}(t)+\mathcal{E}_{a}^{\frac{1}{2}}(t)\mathcal{E}_{1}^{\frac{3}{2}}(t).
\end{equation}
For $I_{8}$, applying integration by parts and $\mathbb{P}u=u$, $I_{8}$ can be rewritten as follows
\begin{align}\label{2.38}
I_{8}&=-\frac{d}{dt}\left\langle\nabla u,G\right\rangle_{\dot{H}^{-s_{1}}\cap\dot{H}^{1}}+\left\langle\nabla u,G_{t}\right\rangle_{\dot{H}^{-s_{1}}\cap\dot{H}^{1}}\nonumber \\
&=-\frac{d}{dt}\left\langle\nabla u,G\right\rangle_{\dot{H}^{-s_{1}}\cap\dot{H}^{1}}+\left\langle\nabla u,2D(u)-u\cdot\nabla G-Q(\nabla u, G)\right\rangle_{\dot{H}^{-s_{1}}\cap\dot{H}^{1}}\nonumber \\
&:=I_{81}+I_{82},
\end{align}
where
\begin{align*}
&I_{81}=-\frac{d}{dt}\left\langle\nabla u,G\right\rangle_{\dot{H}^{-s_{1}}\cap\dot{H}^{1}};\\
&I_{82}=\left\langle\nabla u,2D(u)-u\cdot\nabla G-Q(\nabla u, G)\right\rangle_{\dot{H}^{-s_{1}}\cap\dot{H}^{1}}.
\end{align*}
Using integration by parts, H\"{o}lder's inequality, and Lemma \ref{le5.1}, we get
\begin{align}\label{2.39}
\int_{0}^{t}I_{81}(t')\,dt'&=-\left\langle\nabla u,G\right\rangle_{\dot{H}^{-s_{1}}\cap\dot{H}^{1}}\Big|_{0}^{t}\nonumber \\
&\lesssim\|\nabla u(t)\|_{\dot{H}^{-s_{1}}\cap\dot{H}^{1}}\|G(t)\|_{\dot{H}^{-s_{1}}\cap\dot{H}^{1}}+\|\nabla u_{0}\|_{\dot{H}^{-s_{1}}\cap\dot{H}^{1}}\|G_{0}\|_{\dot{H}^{-s_{1}}\cap\dot{H}^{1}}\nonumber \\
&\lesssim\mathcal{E}_{01}(t).
\end{align}
For $I_{82}$, applying divergence free condition $\diver u=0$ and Lemma \ref{le5.4}, we derive
\begin{align}\label{2.40}
I_{82}&\lesssim\|\nabla u\|_{\dot{H}^{-s_{1}}\cap\dot{H}^{1}}^{2}+\|\nabla^{2} u\|_{\dot{H}^{-s_{1}}}\|u\otimes G\|_{\dot{H}^{-s_{1}}}+\|\nabla^{3} u\|_{L^{2}}\|u\cdot\nabla G\|_{L^{2}}\nonumber \\
&\ \ \ \ + \|\nabla u\|_{\dot{H}^{-s_{1}}\cap\dot{H}^{1}}\|Q(\nabla u, G)\|_{\dot{H}^{-s_{1}}\cap\dot{H}^{1}}\nonumber \\
&\lesssim\|\nabla u\|_{\dot{H}^{-s_{1}}\cap\dot{H}^{1}}^{2}+\|\nabla u\|_{\dot{H}^{-s_{1}}\cap\dot{H}^{2}}\|\nabla u\|_{\dot{H}^{-s_{1}}\cap\dot{H}^{s_{1}}}\|G\|_{\dot{H}^{-s_{1}}}+\|\nabla u\|_{\dot{H}^{2}}\|u\|_{L^{\infty}}\|\nabla G\|_{L^{2}}\nonumber \\
&\ \ \ \ +\|\nabla u\|_{\dot{H}^{-s_{1}}\cap\dot{H}^{1}}\Big(\|\nabla^{2} u\|_{\dot{H}^{-s_{1}}\cap\dot{H}^{s_{1}}}\|G\|_{\dot{H}^{-s_{1}}}+\|\nabla G\|_{L^{3}}\|\nabla u\|_{L^{6}}+\|G\|_{L^{\infty}}\|\nabla^{2} u\|_{L^{2}}\Big)\nonumber \\
&\lesssim\|\nabla u\|_{\dot{H}^{-s_{1}}\cap\dot{H}^{1}}^{2}+\|\nabla u\|_{\dot{H}^{-s_{1}}\cap\dot{H}^{2}}^{2}\|G\|_{\dot{H}^{-s_{1}}}+\|\nabla u\|_{\dot{H}^{2}}\|u\|_{\dot{H}^{1+s_{1}-s_{0}}\cap\dot{H}^{2}}\|\mathbb{P}(\diver  G)\|_{\dot{H}^{s_{1}-s_{0}}\cap\dot{H}^{1}}\nonumber \\
&\ \ \ \ +\|\nabla u\|_{\dot{H}^{-s_{1}}\cap\dot{H}^{1}}\|\nabla u\|_{\dot{H}^{-s_{1}}\cap\dot{H}^{2}}\|G\|_{\dot{H}^{-s_{1}}}+\|\nabla u\|_{\dot{H}^{-s_{1}}\cap\dot{H}^{1}}\| u\|_{\dot{H}^{1+s_{1}-s_{0}}\cap\dot{H}^{2}}\|\mathbb{P}(\diver  G)\|_{\dot{H}^{s_{1}-s_{0}}\cap\dot{H}^{1}}.
\end{align}
Integrating (\ref{2.40}) from $0$ to $t$, we arrive at
\begin{align}\label{2.41}
\int_{0}^{t}I_{82}(t')\,dt'&\lesssim\mathcal{E}_{01}(t)+\mathcal{E}_{01}^{\frac{3}{2}}(t)+\mathcal{E}_{01}(t)\mathcal{E}_{1}^{\frac{1}{2}}(t)
+\mathcal{E}_{1}^{\frac{3}{2}}(t)\nonumber \\
&\lesssim\mathcal{E}_{total}(t)+\mathcal{E}_{total}^{\frac{3}{2}}(t).
\end{align}
Combining with (\ref{2.39}) and (\ref{2.41}), we have
\begin{align}\label{2.42}
\int_{0}^{t}|I_{8}(t')|\,dt'\lesssim\mathcal{E}_{total}(t)+\mathcal{E}_{total}^{\frac{3}{2}}(t).
\end{align}
To establish the estimate of $I_{9}$, using $\diver $ to the second equation of system (\ref{2.4}), one gets
\begin{equation}\label{2.43}
\diver  G_{t}+\diver  (u\cdot\nabla G)+\diver  Q(\nabla u, G)=\Delta u.
\end{equation}
Next, using integration by parts and exploiting the first equation of system (\ref{1.1}) and (\ref{2.43}), we derive
\begin{align}\label{2.44}
I_{9}&=\sum_{k=-s_{1},1}\frac{d}{dt}\int_{\mathbb{R}^{2}}\nabla^{k}\big(\mathbb{P}(\rho u)\big)\nabla^{k}\mathbb{P}(\diver  G)  \,dx-\sum_{k=-s_{1},1}\int_{\mathbb{R}^{2}}\nabla^{k}\big(\mathbb{P}(\rho_{t} u)\big)\nabla^{k}\mathbb{P}(\diver  G)  \,dx \nonumber \\
&\ \ \ \ -\sum_{k=-s_{1},1}\int_{\mathbb{R}^{2}}\nabla^{k}\big(\mathbb{P}(\rho u)\big)\nabla^{k}\mathbb{P}(\diver  G_{t})  \,dx\nonumber \\
&=\sum_{k=-s_{1},1}\frac{d}{dt}\int_{\mathbb{R}^{2}}\nabla^{k}\big(\mathbb{P}(\rho u)\big)\nabla^{k}\mathbb{P}(\diver  G)  \,dx+I_{91}+I_{92},
\end{align}
where
\begin{align*}
&I_{91}=\sum_{k=-s_{1},1}\int_{\mathbb{R}^{2}}\nabla^{k}\big[\mathbb{P}\big((u\cdot\nabla\rho) u\big)\big]\nabla^{k}\mathbb{P}(\diver  G)  \,dx;\\
&I_{92}=\sum_{k=-s_{1},1}\int_{\mathbb{R}^{2}}\nabla^{k}\big(\mathbb{P}(\rho u)\big)\nabla^{k}\mathbb{P}\big[\Delta u-\diver  (u\cdot\nabla G)-\diver  Q(\nabla u, G)\big]\,dx.
\end{align*}
Next, using H\"{o}lder's inequality, Lemmas \ref{le5.1}--\ref{le5.2} and Lemma \ref{le5.4}, we estimate $I_{91}$, $I_{92}$ as follows
\begin{align}\label{2.45}
I_{91}&\lesssim\|(u\cdot\nabla\rho) u\|_{\dot{H}^{-s_{1}}}\|\mathbb{P}(\diver  G)\|_{\dot{H}^{-s_{1}}}+\|(u\cdot\nabla\rho) u\|_{\dot{H}^{1}}\|\mathbb{P}(\diver  G)\|_{\dot{H}^{1}}\nonumber \\
&\lesssim\|\diver(\rho u\otimes u)-\rho u\cdot\nabla u\|_{\dot{H}^{-s_{1}}}\|\mathbb{P}(\diver  G)\|_{\dot{H}^{-s_{1}}}+\Big(\|u\|_{L^{\infty}}^{2}\|\nabla^{2}\rho\|_{L^{2}}\nonumber \\
&\ \ \ \ +\|u\|_{L^{\infty}}\|\nabla(u\cdot\nabla\rho)\|_{L^{2}}+\|\nabla u\|_{L^{6}}\|u\|_{L^{\infty}}\|\nabla\rho\|_{L^{3}}\Big)\|\mathbb{P}(\diver  G)\|_{\dot{H}^{1}}\nonumber \\
&\lesssim\Big(\||\nabla|^{1-s_{1}}(\rho u\otimes u)\|_{L^{2}}+\|\rho u\cdot\nabla u\|_{\dot{H}^{-s_{1}}}\Big)\|\mathbb{P}(\diver  G)\|_{\dot{H}^{-s_{1}}}+\|u\|_{\dot{H}^{1+s_{1}-s_{0}}\cap\dot{H}^{2}}^{2}\|\rho\|_{H^{2}}\|\mathbb{P}(\diver  G)\|_{\dot{H}^{1}}\nonumber \\
&\lesssim\Big(\||\nabla|^{1-s_{1}}(\rho u)\|_{L^{2}}\|u\|_{L^{\infty}}+\|\rho u\|_{L^{3}}\||\nabla|^{1-s_{1}} u\|_{L^{6}}+\|\nabla(\rho u)\|_{\dot{H}^{-s_{1}}}\|\nabla u\|_{\dot{H}^{-s_{1}}}\Big)\|\mathbb{P}(\diver  G)\|_{\dot{H}^{-s_{1}}}\nonumber \\
&\ \ \ \ +\|u\|_{\dot{H}^{1+s_{1}-s_{0}}\cap\dot{H}^{2}}^{2}\|\rho\|_{H^{2}}\|\mathbb{P}(\diver  G)\|_{\dot{H}^{1}}\nonumber \\
&\lesssim\Big(\|u\|_{\dot{H}^{1+s_{1}-s_{0}}\cap\dot{H}^{2}}^{2}+\|u\|_{\dot{H}^{1+s_{1}-s_{0}}\cap\dot{H}^{2}}\|\nabla u\|_{\dot{H}^{-s_{1}}}\Big)\|\rho\|_{H^{2}}\|\mathbb{P}(\diver  G)\|_{\dot{H}^{-s_{1}}}+\|u\|_{\dot{H}^{1+s_{1}-s_{0}}\cap\dot{H}^{2}}^{2}\|\rho\|_{H^{2}}\|\mathbb{P}(\diver  G)\|_{\dot{H}^{1}}
\end{align}
and
\begin{align}\label{2.46}
I_{92}&\lesssim\|\nabla(\rho u)\|_{\dot{H}^{-s_{1}}}\|\nabla u\|_{\dot{H}^{-s_{1}}}+\|\nabla(\rho u)\|_{L^{2}}\|\nabla^{3} u\|_{L^{2}}+\|\nabla(\rho u)\|_{\dot{H}^{-s_{1}}}\|u\cdot\nabla G\|_{\dot{H}^{-s_{1}}}+\|\rho u\|_{\dot{H}^{2}}\|\nabla(u\cdot\nabla G)\|_{L^{2}}\nonumber \\
&\ \ \ \ +\|\nabla(\rho u)\|_{\dot{H}^{-s_{1}}}\|Q(\nabla u, G)\|_{\dot{H}^{-s_{1}}}+\|\rho u\|_{\dot{H}^{2}}\|\nabla (Q(\nabla u, G))\|_{L^{2}}\nonumber \\
&\lesssim\Big(\||\nabla|^{1-s_{1}}\rho\|_{L^{2}}\|u\|_{L^{\infty}}+\|\rho\|_{L^{\infty}}\||\nabla|^{1-s_{1}} u\|_{L^{2}}\Big)\Big(\|\nabla u\|_{\dot{H}^{-s_{1}}}+\|\nabla u\|_{\dot{H}^{-s_{1}}\cap\dot{H}^{s_{1}}}\|\nabla G\|_{\dot{H}^{-s_{1}}}\nonumber \\
&\ \ \ \ +\|\nabla^{2} u\|_{\dot{H}^{-s_{1}}\cap\dot{H}^{s_{1}}}\|G\|_{\dot{H}^{-s_{1}}}\Big)+\Big(\|\rho\|_{L^{\infty}}\|\nabla^{2} u\|_{L^{2}}+\|\nabla^{2} \rho\|_{L^{2}}\|u\|_{L^{\infty}}\Big)\Big(\|\nabla^{2} u\|_{L^{2}}\|G\|_{L^{\infty}}+\|\nabla u\|_{L^{6}}\|\nabla G\|_{L^{3}}\nonumber \\
&\ \ \ \ +\|u\|_{L^{\infty}}\|\nabla^{2} G\|_{L^{2}}\Big)+\Big(\|\nabla\rho\|_{L^{2}}\|u\|_{L^{\infty}}+\|\rho\|_{L^{3}}\|\nabla u\|_{L^{6}}\Big)\|\nabla u\|_{\dot{H}^{2}}\nonumber \\
&\lesssim\|\rho\|_{H^{2}}\Big[\|u\|_{\dot{H}^{1-s_{1}}\cap\dot{H}^{2}}^{2}+\|u\|_{\dot{H}^{1+s_{1}-s_{0}}\cap\dot{H}^{2}}\|\nabla u\|_{\dot{H}^{-s_{1}}}(1+\|\nabla G\|_{\dot{H}^{-s_{1}}})+\|u\|_{\dot{H}^{1+s_{1}-s_{0}}\cap\dot{H}^{2}}^{2}\|\nabla G\|_{\dot{H}^{-s_{1}}}\nonumber \\
&\ \ \ \ +\|\nabla u\|_{\dot{H}^{-s_{1}}\cap\dot{H}^{2}}^{2}\|G\|_{\dot{H}^{-s_{1}}}\Big]+\|\rho\|_{H^{2}}\| u\|_{\dot{H}^{1+s_{1}-s_{0}}\cap\dot{H}^{2}}^{2}\|G\|_{\dot{H}^{1-s_{1}}\cap\dot{H}^{2}}+\|\rho\|_{H^{2}}\|u\|_{\dot{H}^{1-s_{1}}\cap\dot{H}^{2}}\|\nabla u\|_{\dot{H}^{2}}.
\end{align}
Substituting (\ref{2.45}) and (\ref{2.46}) into (\ref{2.44}) and integrating (\ref{2.44}) from $0$ to $t$, we get
\begin{align}\label{2.47}
\int_{0}^{t}|I_{9}(t')|\,dt'&\lesssim\sup_{0\leq t'\leq t}\Big(\|\rho u\|_{\dot{H}^{-s_{1}}}\|\mathbb{P}(\diver  G)\|_{\dot{H}^{-s_{1}}}+\|\rho u\|_{\dot{H}^{1}}\|\mathbb{P}(\diver  G)\|_{\dot{H}^{1}}\Big)+\mathcal{E}_{1}(t)\mathcal{E}_{a}^{\frac{1}{2}}(t)\mathcal{E}_{0}^{\frac{1}{2}}(t)\nonumber \\
&\ \ \ \ +\mathcal{E}_{a}^{\frac{1}{2}}(t)\mathcal{E}_{0}(t)+\mathcal{E}_{a}^{\frac{1}{2}}(t)\mathcal{E}_{1}(t)+\mathcal{E}_{a}^{\frac{1}{2}}(t)\mathcal{E}_{1}(t)
\mathcal{E}_{0}^{\frac{1}{2}}(t)+\mathcal{E}_{a}^{\frac{1}{2}}(t)\mathcal{E}_{0}^{\frac{3}{2}}(t)\nonumber \\
&\lesssim\sup_{0\leq t'\leq t}\Big(\|\nabla\rho\|_{\dot{H}^{-s_{1}}\cap\dot{H}^{s_{1}}}\|u\|_{\dot{H}^{-s_{1}}}\|\mathbb{P}(\diver  G)\|_{\dot{H}^{-s_{1}}}+\|\rho\|_{L^{\infty}\cap\dot{H}^{1}}\|u\|_{L^{\infty}\cap\dot{H}^{1}}\|\mathbb{P}(\diver  G)\|_{\dot{H}^{1}}\Big)\nonumber \\
&\ \ \ \ +\mathcal{E}_{a}^{\frac{1}{2}}(t)\Big(\mathcal{E}_{1}(t)\mathcal{E}_{0}^{\frac{1}{2}}(t)
+\mathcal{E}_{0}(t)+\mathcal{E}_{1}(t)+\mathcal{E}_{0}^{\frac{3}{2}}(t)\Big)\nonumber \\
&\lesssim\mathcal{E}_{a}^{\frac{1}{2}}(t)\Big(\mathcal{E}_{0}(t)+\mathcal{E}_{1}(t)+\mathcal{E}_{0}^{\frac{3}{2}}(t)+\mathcal{E}_{1}^{\frac{3}{2}}(t)\Big).
\end{align}

Putting all the estimates of (\ref{2.33}), (\ref{2.35}), (\ref{2.37}), (\ref{2.42}) and (\ref{2.47}) together, and integrating (\ref{2.32}) from $0$ to $t$, we get

\begin{align}\label{2.48}
\mathcal{E}_{02}(t)&=\int_{0}^{t}\|\mathbb{P}(\diver  G)(t')\|_{\dot{H}^{-s_{1}}\cap\dot{H}^{1}}^{2}\,dt'\nonumber \\
&\lesssim\mathcal{E}_{01}^{\frac{1}{2}}(t)\mathcal{E}_{02}^{\frac{1}{2}}(t)+\mathcal{E}_{0}^{\frac{3}{2}}(t)
+\mathcal{E}_{1}^{\frac{3}{2}}(t)+\mathcal{E}_{1}(t)\mathcal{E}_{a}^{\frac{1}{2}}(t)\mathcal{E}_{0}^{\frac{1}{2}}(t)+\mathcal{E}_{total}(t)+\mathcal{E}_{total}^{\frac{3}{2}}(t)\nonumber \\
&\lesssim\mathcal{E}_{0}(0)+\mathcal{E}_{total}(t)+\mathcal{E}_{total}^{\frac{3}{2}}(t)+\mathcal{E}_{total}^{\frac{9}{4}}(t).
\end{align}

From the estimates of (\ref{2.30}) and (\ref{2.48}), we finally derive that
\begin{align}\label{2.49}
\mathcal{E}_{0}(t)\lesssim\mathcal{E}_{0}(0)+\mathcal{E}_{total}(t)+\mathcal{E}_{total}^{\frac{3}{2}}(t)+\mathcal{E}_{total}^{\frac{9}{4}}(t).
\end{align}

Therefore, the proof of Lemma \ref{le2.4} is completed.

\end{proof}

\subsection{Estimates of Fractional Time-weighted Energy $\mathcal{E}_{1}(t)$}
\ \ \ \\
In this subsection, we are ready to establish the estimate of fractional time-weighted energy $\mathcal{E}_{1}(t)$. Precisely, we prove the following result.
\begin{Lemma}\label{le2.5}
Assume that $\mathcal{E}_{1}(t)$ is defined as in (\ref{2.6}). Then the following estimate is given
\begin{equation}\label{2.50}
\mathcal{E}_{1}(t)\lesssim\mathcal{E}_{1}(0)+\mathcal{E}_{total}(t)+\mathcal{E}_{total}^{\frac{3}{2}}(t)
\end{equation}
for any $t>0$.
\end{Lemma}
\begin{proof}
Similar to the proof strategy of Lemma \ref{le2.4}, we still divide fractional time-weighted energy $\mathcal{E}_{1}(t)$ into the following two parts:
\begin{align*}
&\mathcal{E}_{11}(t)=\sup_{0\leq t'\leq t}(1+t')^{1+2s_{1}-s_{0}}(\|u(t')\|_{\dot{H}^{1+s_{1}-s_{0}}\cap\dot{H}^{2}}^{2}+\|\mathbb{P}\diver  G(t')\|_{\dot{H}^{s_{1}-s_{0}}\cap\dot{H}^{1}}^{2})\nonumber \\
&\ \ \ \ \ \ \ \ \ \ \ \ \ \ +\int_{0}^{t}(1+t')^{1+2s_{1}-s_{0}}\|\nabla u(t')\|_{\dot{H}^{1+s_{1}-s_{0}}\cap\dot{H}^{2}}^{2}\,dt', \\
&\mathcal{E}_{12}(t)=\int_{0}^{t}(1+t')^{1+2s_{1}-s_{0}}\|\mathbb{P}\diver  G(t')\|_{\dot{H}^{1+s_{1}-s_{0}}\cap\dot{H}^{1}}^{2}\,dt'.
\end{align*}
\textbf{Step 1.} The estimate of $\mathcal{E}_{11}(t)$.\\
Applying the operators $\nabla^{1+k}, \nabla^{k}\mathbb{P}{\rm div}, (k=s_{1}-s_{0},1)$ to $(\ref{2.20})_{1}, (\ref{2.20})_{2}$, respectively. We get
\begin{equation}\label{2.51}
\begin{cases}
\nabla^{k+1}u_{t}-\nabla^{k+1}\Delta u+\nabla^{k+1}(u\cdot\nabla u)+\nabla^{k+1}\big[\frac{\rho}{\rho+1}(\Delta u+\diver  G)\big]+\nabla^{k+1}\nabla\tilde{p}=\nabla^{k+1}\diver  G,\\
\nabla^{k}\mathbb{P}\diver  G_{t}+\nabla^{k}\mathbb{P}\diver (u\cdot\nabla G)+\nabla^{k}\mathbb{P}\diver  Q(\nabla u, \mathbb{F}, \tilde{\rho})=\nabla^{k}\Delta u.
\end{cases}
\end{equation}
Then taking inner product with $(\nabla^{k+1}u,\nabla^{k}\mathbb{P}\diver G)$ for $(\ref{2.51})_{1}, (\ref{2.51})_{2}$, respectively. Multiplying the time weight $(1+t')^{1+2s_{1}-s_{0}}$, one gets
\begin{align}\label{2.52}
&\frac{1}{2}\frac{d}{dt}\left[(1+t')^{1+2s_{1}-s_{0}}(\|u(t')\|_{\dot{H}^{1+s_{1}-s_{0}}\cap\dot{H}^{2}}^{2}+\|\mathbb{P}\diver  G(t')\|_{\dot{H}^{s_{1}-s_{0}}\cap\dot{H}^{1}}^{2})\right]\nonumber \\
&+(1+t')^{1+2s_{1}-s_{0}}\|\nabla u(t')\|_{\dot{H}^{1+s_{1}-s_{0}}\cap\dot{H}^{2}}^{2}
=J_{1}+J_{2}+J_{3}+J_{4}+J_{5}+J_{6},
\end{align}
where
\begin{align*}
&J_{1}=\frac{1+2s_{1}-s_{0}}{2}(1+t')^{2s_{1}-s_{0}}(\|u(t')\|_{\dot{H}^{1+s_{1}-s_{0}}\cap\dot{H}^{2}}^{2}+\|\mathbb{P}\diver  G(t')\|_{\dot{H}^{s_{1}-s_{0}}\cap\dot{H}^{1}}^{2});\\
&J_{2}=(1+t')^{1+2s_{1}-s_{0}}\left\langle\diver  G,u\right\rangle_{\dot{H}^{1+s_{1}-s_{0}}\cap\dot{H}^{2}}+(1+t')^{1+2s_{1}-s_{0}}\left\langle\Delta u,\mathbb{P}\diver  G\right\rangle_{\dot{H}^{s_{1}-s_{0}}\cap\dot{H}^{1}};\\
&J_{3}=-(1+t')^{1+2s_{1}-s_{0}}\left\langle u\cdot\nabla u,u\right\rangle_{\dot{H}^{1+s_{1}-s_{0}}\cap\dot{H}^{2}};\\
&J_{4}=-(1+t')^{1+2s_{1}-s_{0}}\left\langle\mathbb{P}\diver  (u\cdot\nabla G),\mathbb{P}\diver  G\right\rangle_{\dot{H}^{s_{1}-s_{0}}\cap\dot{H}^{1}};\\
&J_{5}=-(1+t')^{1+2s_{1}-s_{0}}\left\langle g(\rho)(\diver  G+\Delta u),u\right\rangle_{\dot{H}^{1+s_{1}-s_{0}}\cap\dot{H}^{2}};\\
&J_{6}=-(1+t')^{1+2s_{1}-s_{0}}\left\langle\mathbb{P}\diver  Q(\nabla u,G),\mathbb{P}\diver  G \right\rangle_{\dot{H}^{s_{1}-s_{0}}\cap\dot{H}^{1}}.
\end{align*}
For $J_{1}$, making use of H\"{o}lder's inequality and Lemma \ref{le5.1}, one gets
\begin{align}\label{2.53}
\int_{0}^{t}|J_{1}(t')|\,dt'&\lesssim\int_{0}^{t}\|u(t')\|_{\dot{H}^{1-s_{1}}}^{\frac{2}{1+2s_{1}-s_{0}}}\cdot\Big[
(1+t')^{1+2s_{1}-s_{0}}\|\nabla u(t')\|_{\dot{H}^{1+s_{1}-s_{0}}}^{2}\Big]^{\frac{2s_{1}-s_{0}}{1+2s_{1}-s_{0}}}\,dt'\nonumber \\
&\ \ \ \ +\int_{0}^{t}\|\mathbb{P}\diver  G(t')\|_{\dot{H}^{-s_{1}}}^{\frac{2}{1+2s_{1}-s_{0}}}\cdot\Big[
(1+t')^{1+2s_{1}-s_{0}}\|\mathbb{P}\diver  G(t')\|_{\dot{H}^{1+s_{1}-s_{0}}}^{2}\Big]^{\frac{2s_{1}-s_{0}}{1+2s_{1}-s_{0}}}\,dt' \nonumber \\
&\ \ \ \ +\int_{0}^{t}(1+t')^{1+2s_{1}-s_{0}}(\|\nabla u(t')\|_{\dot{H}^{1+s_{1}-s_{0}}\cap\dot{H}^{2}}^{2}+\|\mathbb{P}\diver  G(t')\|_{\dot{H}^{1}}^{2})\,dt' \nonumber \\
&\lesssim\mathcal{E}_{0}^{\frac{1}{1+2s_{1}-s_{0}}}(t)\mathcal{E}_{1}^{\frac{2s_{1}-s_{0}}{1+2s_{1}-s_{0}}}(t)+\mathcal{E}_{1}(t),
\end{align}
where $s_{0}, s_{1}$ are required to satisfy $\frac{\sqrt{7}-1}{3}<s_{1}<1, \max\{s_{1},\frac{3s_{1}^{2}-2s_{1}+2}{3s_{1}}\}<s_{0}<2s_{1}$.\\
For $J_{2}$, similar to (\ref{2.23}), applying integration by parts and divergence free condition $\diver u=0$, we can obtain
\begin{equation}\label{2.54}
J_{2}=0.
\end{equation}
For $J_{3}$, using Lemma \ref{le5.1} and Lemma \ref{le5.3}, we derive that
\begin{align}\label{2.55}
J_{3}&\lesssim(1+t')^{1+2s_{1}-s_{0}}\Big(\|[|\nabla|^{s_{1}-s_{0}},u\cdot\nabla]\nabla u\|_{L^{2}}\|u\|_{\dot{H}^{1+s_{1}-s_{0}}}+\||\nabla|^{s_{1}-s_{0}}(\nabla u\cdot\nabla u)\|_{L^{2}}\|u\|_{\dot{H}^{1+s_{1}-s_{0}}}\Big)\nonumber \\
&\ \ \ \ +(1+t')^{1+2s_{1}-s_{0}}\|\nabla u\|_{L^{\infty}}\|\nabla^{2} u\|_{L^{2}}^{2}\nonumber \\
&\lesssim(1+t')^{1+2s_{1}-s_{0}}\|\nabla^{2} u\|_{\dot{H}^{-s_{0}+s_{1}}\cap\dot{H}^{s_{0}-s_{1}}}\| u\|_{\dot{H}^{1+s_{1}-s_{0}}}^{2}+(1+t')^{1+2s_{1}-s_{0}}\|\nabla u\|_{\dot{H}^{1+s_{1}-s_{0}}\cap\dot{H}^{2}}\| u\|_{\dot{H}^{2}}^{2}.
\end{align}
Integrating (\ref{2.55}) from $0$ to $t$, we have
\begin{align}\label{2.56}
&\int_{0}^{t}|J_{3}(t')|\,dt'\nonumber \\
&\lesssim\sup_{0\leq t'\leq t}(1+t')^{1+2s_{1}-s_{0}}\|u\|_{\dot{H}^{1+s_{1}-s_{0}}\cap\dot{H}^{2}}^{2}\int_{0}^{t}(1+t')^{-\frac{1+2s_{1}-s_{0}}{2}}(1+t')^{\frac{1+2s_{1}-s_{0}}{2}}\|\nabla u\|_{\dot{H}^{1+s_{1}-s_{0}}\cap\dot{H}^{2}}\,dt'\nonumber \\
&\lesssim\mathcal{E}_{1}^{\frac{3}{2}}(t).
\end{align}
For $J_{4}$, we utilize Lemma \ref{le2.2} and rewrite $J_{4}=J_{41}+J_{42}+J_{43}$, where
\begin{align*}
&J_{41}=-(1+t')^{1+2s_{1}-s_{0}}\left\langle\mathbb{P}(u\cdot\nabla \mathbb{P}\diver  G),\mathbb{P}\diver  G\right\rangle_{\dot{H}^{s_{1}-s_{0}}\cap\dot{H}^{1}};\\
&J_{42}=-(1+t')^{1+2s_{1}-s_{0}}\left\langle\mathbb{P}(\nabla u\cdot\nabla G),\mathbb{P}\diver  G\right\rangle_{\dot{H}^{s_{1}-s_{0}}\cap\dot{H}^{1}};\\
&J_{43}=-(1+t')^{1+2s_{1}-s_{0}}\left\langle\mathbb{P}(\nabla u\cdot\nabla (-\Delta)^{-1}\diver \diver G),\mathbb{P}\diver  G\right\rangle_{\dot{H}^{s_{1}-s_{0}}\cap\dot{H}^{1}}.
\end{align*}
Considering the fact $\mathbb{P}\mathbb{P}=\mathbb{P}$, utilizing integration by parts, divergence free condition and Lemmas \ref{le5.2}- \ref{le5.4},, we can deal with the first two terms $J_{41}, J_{42}$ as follows:
\begin{align}\label{2.57}
J_{41}&=-(1+t')^{1+2s_{1}-s_{0}}\left\langle(u\cdot\nabla \mathbb{P}\diver  G),\mathbb{P}\mathbb{P}\diver  G\right\rangle_{\dot{H}^{s_{1}-s_{0}}\cap\dot{H}^{1}}\nonumber \\
&=-(1+t')^{1+2s_{1}-s_{0}}\left\langle(u\cdot\nabla \mathbb{P}\diver  G),\mathbb{P}\diver  G\right\rangle_{\dot{H}^{s_{1}-s_{0}}\cap\dot{H}^{1}}\nonumber \\
&\lesssim(1+t')^{1+2s_{1}-s_{0}}\Big(\|[|\nabla|^{s_{1}-s_{0}},u\cdot\nabla]\mathbb{P}\diver  G\|_{L^{2}}\|\mathbb{P}\diver  G\|_{\dot{H}^{s_{1}-s_{0}}}+\|\nabla u\|_{L^{\infty}}\|\mathbb{P}\diver  G\|_{\dot{H}^{1}}^{2}\Big)\nonumber \\
&\lesssim(1+t')^{1+2s_{1}-s_{0}}\Big(\|\nabla^{2} u\|_{\dot{H}^{-s_{0}+s_{1}}\cap\dot{H}^{s_{0}-s_{1}}}\|\mathbb{P}\diver  G\|_{\dot{H}^{s_{1}-s_{0}}}^{2}+\|\nabla^{2} u\|_{\dot{H}^{-s_{0}+s_{1}}\cap\dot{H}^{1}}\|\mathbb{P}\diver  G\|_{\dot{H}^{1}}^{2}\Big)
\end{align}
and
\begin{align}\label{2.58}
J_{42}&=-(1+t')^{1+2s_{1}-s_{0}}\Big(\left\langle\mathbb{P}\diver(\nabla u\otimes G),\mathbb{P}\diver  G\right\rangle_{\dot{H}^{s_{1}-s_{0}}}+\left\langle\mathbb{P}(\nabla u\cdot \nabla G),\mathbb{P}\diver  G\right\rangle_{\dot{H}^{1}}\Big)\nonumber \\
&\lesssim(1+t')^{1+2s_{1}-s_{0}}\Big(\|\nabla u\otimes G\|_{\dot{H}^{s_{1}-s_{0}}}\|\nabla\mathbb{P}\diver  G\|_{\dot{H}^{s_{1}-s_{0}}}+\|\nabla u\cdot \nabla G\|_{\dot{H}^{1}}\|\mathbb{P}\diver  G\|_{\dot{H}^{1}}\Big)\nonumber \\
&\lesssim(1+t')^{1+2s_{1}-s_{0}}\Big[\|\nabla^{2} u\|_{\dot{H}^{-s_{0}+s_{1}}\cap\dot{H}^{s_{0}-s_{1}}}\|  G\|_{\dot{H}^{s_{1}-s_{0}}}\|\nabla\mathbb{P}\diver  G\|_{\dot{H}^{s_{1}-s_{0}}}\nonumber \\
&\ \ \ \ +(\|\nabla^{2} u\|_{L^{4}}\|\nabla G\|_{L^{4}}+\|\nabla u\|_{L^{\infty}}\|\nabla^{2} G\|_{L^{2}})\|\mathbb{P}\diver  G\|_{\dot{H}^{1}}\Big].
\end{align}
Integrating (\ref{2.57}), (\ref{2.58}) from $0$ to $t$, we get
\begin{align}\label{2.59}
&\int_{0}^{t}|J_{41}(t')|\,dt'\nonumber \\
&\lesssim\sup_{0\leq t'\leq t}(1+t')^{1+2s_{1}-s_{0}}\|\mathbb{P}\diver  G\|_{\dot{H}^{s_{1}-s_{0}}\cap\dot{H}^{1}}^{2}\int_{0}^{t}\|\nabla^{2} u\|_{\dot{H}^{-s_{0}+s_{1}}\cap\dot{H}^{1}}\,dt'\nonumber \\
&\lesssim\mathcal{E}_{1}^{\frac{3}{2}}(t)
\end{align}
and
\begin{align}\label{2.60}
&\int_{0}^{t}|J_{42}(t')|\,dt'\nonumber \\
&\lesssim\sup_{0\leq t'\leq t}\|G\|_{\dot{H}^{-s_{1}}\cap\dot{H}^{2}}\int_{0}^{t}(1+t')^{1+2s_{1}-s_{0}}\|\nabla^{2} u\|_{\dot{H}^{s_{1}-s_{0}}\cap\dot{H}^{1}}\|\mathbb{P}\diver  G\|_{\dot{H}^{1+s_{1}-s_{0}}}\,dt'\nonumber \\
&\ \ \ \ +\sup_{0\leq t'\leq t}(1+t')^{1+2s_{1}-s_{0}}\|\mathbb{P}\diver  G\|_{\dot{H}^{s_{1}-s_{0}}\cap\dot{H}^{1}}^{2}\int_{0}^{t}\|\nabla^{2} u\|_{\dot{H}^{s_{1}-s_{0}}\cap\dot{H}^{1}}\,dt'\nonumber \\
&\lesssim\mathcal{E}_{0}^{\frac{1}{2}}(t)\mathcal{E}_{1}(t)+\mathcal{E}_{1}^{\frac{3}{2}}(t).
\end{align}
Next, considering the boundedness of Riesz operator $\mathcal{R}_{i}=(-\Delta)^{-\frac{1}{2}}\nabla_{i}$ from $L^{2}$ into itself, we can deal with the last term $J_{43}$ as follows:
\begin{align}\label{2.61}
J_{43}&=-(1+t')^{1+2s_{1}-s_{0}}\Big(\left\langle\nabla u\cdot\nabla (-\Delta)^{-1}\diver \diver G,\mathbb{P}\diver  G\right\rangle_{\dot{H}^{s_{1}-s_{0}}}\nonumber \\
&\ \ \ \ +\left\langle\nabla u\cdot\nabla (-\Delta)^{-1}\diver \diver G,\mathbb{P}\diver  G\right\rangle_{\dot{H}^{1}}\Big)\nonumber \\
&\lesssim(1+t')^{1+2s_{1}-s_{0}}\Big[\|\nabla^{2} u\|_{\dot{H}^{-s_{0}+s_{1}}\cap\dot{H}^{s_{0}-s_{1}}}\|(-\Delta)^{-1}\diver \diver G\|_{\dot{H}^{1+s_{1}-s_{0}}}\|\mathbb{P}\diver  G\|_{\dot{H}^{s_{1}-s_{0}}}\nonumber \\
&\ \ \ \ +(\|\nabla^{2} u\|_{L^{4}}\|\nabla(-\Delta)^{-1}\diver \diver G\|_{L^{4}}+\|\nabla u\|_{L^{\infty}}\|\nabla^{2}(-\Delta)^{-1}\diver \diver G\|_{L^{2}})\|\mathbb{P}\diver  G\|_{\dot{H}^{1}}\Big]\nonumber \\
&\lesssim(1+t')^{1+2s_{1}-s_{0}}\Big(\|\nabla^{2} u\|_{\dot{H}^{-s_{0}+s_{1}}\cap\dot{H}^{s_{0}-s_{1}}}\|\mathbb{P}\diver  G\|_{\dot{H}^{s_{1}-s_{0}}}^{2}+\|\nabla^{2} u\|_{\dot{H}^{-s_{0}+s_{1}}\cap\dot{H}^{1}}\|\mathbb{P}\diver  G\|_{\dot{H}^{s_{1}-s_{0}}\cap\dot{H}^{1}}^{2}\Big).
\end{align}
Integrating (\ref{2.61}) from $0$ to $t$, we have
\begin{align}\label{2.62}
\int_{0}^{t}|J_{43}(t')|\,dt'&\lesssim\sup_{0\leq t'\leq t}(1+t')^{1+2s_{1}-s_{0}}\|\mathbb{P}\diver  G\|_{\dot{H}^{s_{1}-s_{0}}\cap\dot{H}^{1}}^{2}\int_{0}^{t}\|\nabla^{2} u\|_{\dot{H}^{s_{1}-s_{0}}\cap\dot{H}^{1}}\,dt'\nonumber \\
&\lesssim\mathcal{E}_{1}^{\frac{3}{2}}(t).
\end{align}
For $J_{5}$, using H\"{o}lder's inequality, Lemma \ref{le5.1} and Lemma \ref{le5.4}, we get
\begin{align}\label{2.63}
J_{5}&\lesssim(1+t')^{1+2s_{1}-s_{0}}\Big(\|g(\rho)(\Delta u+\diver G)\|_{\dot{H}^{s_{1}-s_{0}}}\|\nabla u\|_{\dot{H}^{1+s_{1}-s_{0}}}+\|g(\rho)(\Delta u+\diver G)\|_{\dot{H}^{1}}\|\nabla u\|_{\dot{H}^{2}}\Big)\nonumber \\
&\lesssim(1+t')^{1+2s_{1}-s_{0}}\Big[(\|g(\rho)\Delta u\|_{\dot{H}^{s_{1}-s_{0}}}+\|g(\rho)\diver G\|_{\dot{H}^{s_{1}-s_{0}}})\|\nabla u\|_{\dot{H}^{1+s_{1}-s_{0}}}+\|g(\rho)(\Delta u+\diver G)\|_{\dot{H}^{1}}\|\nabla u\|_{\dot{H}^{2}}\Big]\nonumber \\
&\lesssim(1+t')^{1+2s_{1}-s_{0}}\Big[(\|\nabla \rho\|_{\dot{H}^{-s_{0}+s_{1}}\cap\dot{H}^{s_{0}-s_{1}}}\| \Delta u\|_{\dot{H}^{s_{1}-s_{0}}}+\|\vert\nabla\vert^{\sigma}\mathbb{P}\diver  G\|_{\dot{H}^{-r_{1}}\cap \dot{H}^{r_{1}}}\|\rho\|_{\dot{H}^{s_{1}-s_{0}}})\|\nabla u\|_{\dot{H}^{1+s_{1}-s_{0}}}\nonumber \\
&\ \ \ \ +(\|\rho\|_{L^{\infty}}\|\nabla(\Delta u+\diver G)\|_{L^{2}}+\|\nabla\rho\|_{L^{3}}\|\Delta u+\diver G\|_{L^{6}})\|\nabla u\|_{\dot{H}^{2}}\Big]\nonumber \\
&\lesssim(1+t')^{1+2s_{1}-s_{0}}\Big[(\|\rho\|_{H^{2}}\|\nabla u\|_{\dot{H}^{1+s_{1}-s_{0}}}+\|\rho\|_{\dot{H}^{s_{1}-s_{0}}}\|\mathbb{P}\diver  G\|_{\dot{H}^{1+s_{1}-s_{0}}\cap\dot{H}^{1}})\|\nabla u\|_{\dot{H}^{1+s_{1}-s_{0}}}\nonumber \\
&\ \ \ \ +\|\rho\|_{H^{2}}(\|\nabla u\|_{\dot{H}^{2}}+\|\mathbb{P}\diver  G\|_{\dot{H}^{1}}+\|\nabla u\|_{\dot{H}^{1+s_{1}-s_{0}}\cap\dot{H}^{2}}+\|\mathbb{P}\diver  G\|_{\dot{H}^{1+s_{1}-s_{0}}\cap\dot{H}^{1}})\|\nabla u\|_{\dot{H}^{2}}\Big].
\end{align}
Integrating (\ref{2.63}) from $0$ to $t$, we have
\begin{align}\label{2.64}
\int_{0}^{t}|J_{5}(t')|\,dt'&\lesssim\sup_{0\leq t'\leq t}\|\rho\|_{H^{2}\cap\dot{H}^{s_{1}-s_{0}}}\Big(\int_{0}^{t}(1+t')^{1+2s_{1}-s_{0}}\|\nabla u\|_{\dot{H}^{1+s_{1}-s_{0}}\cap\dot{H}^{2}}^{2}\,dt'\nonumber \\
&\ \ \ \ +\int_{0}^{t}(1+t')^{1+2s_{1}-s_{0}}\|\mathbb{P}\diver  G\|_{\dot{H}^{1+s_{1}-s_{0}}\cap\dot{H}^{1}}\|\nabla u\|_{\dot{H}^{1+s_{1}-s_{0}}\cap\dot{H}^{2}}\,dt'\Big)\nonumber \\
&\lesssim\mathcal{E}_{a}^{\frac{1}{2}}(t)\mathcal{E}_{1}(t).
\end{align}
Similarly, the term $J_{6}$ can be bounded by
\begin{align}\label{2.65}
J_{6}&\lesssim(1+t')^{1+2s_{1}-s_{0}}\Big(\|Q(\nabla u,G)\|_{\dot{H}^{s_{1}-s_{0}}}\|\nabla \mathbb{P}\diver  G \|_{\dot{H}^{s_{1}-s_{0}}}+\|\diver (Q(\nabla u,G))\|_{\dot{H}^{1}}\|\mathbb{P}\diver  G\|_{\dot{H}^{1}}\Big)\nonumber \\
&\lesssim(1+t')^{1+2s_{1}-s_{0}}\Big[\|\nabla G \|_{\dot{H}^{s_{1}-s_{0}}\cap\dot{H}^{s_{0}-s_{1}}}\|\nabla u \|_{\dot{H}^{s_{1}-s_{0}}}\|\mathbb{P}\diver  G \|_{\dot{H}^{1+s_{1}-s_{0}}}+(\|\nabla u\|_{L^{\infty}}\|\nabla^{2}G\|_{L^{2}}+\|\nabla^{3}u\|_{L^{2}}\|G\|_{L^{\infty}}\nonumber \\
&\ \ \ \ +\|\nabla^{2}u\|_{L^{4}}\|\nabla G\|_{L^{4}})\|\mathbb{P}\diver  G\|_{\dot{H}^{1}}\Big]\nonumber \\
&\lesssim(1+t')^{1+2s_{1}-s_{0}}\Big(\|\nabla G \|_{\dot{H}^{s_{1}-s_{0}}\cap\dot{H}^{1}}\|\nabla u \|_{\dot{H}^{s_{1}-s_{0}}}\|\mathbb{P}\diver  G \|_{\dot{H}^{1+s_{1}-s_{0}}}\nonumber\\
&\qquad\qquad\qquad\qquad +\|\nabla u\|_{\dot{H}^{s_{1}-s_{0}}\cap\dot{H}^{2}}\|\nabla G \|_{\dot{H}^{s_{1}-s_{0}}\cap\dot{H}^{1}}\|\mathbb{P}\diver  G\|_{\dot{H}^{1}}\Big).
\end{align}
Integrating (\ref{2.65}) from $0$ to $t$, we get
\begin{align}\label{2.66}
\int_{0}^{t}|J_{6}(t')|\,dt'\lesssim\mathcal{E}_{1}^{\frac{3}{2}}(t).
\end{align}
Finally, integrating (\ref{2.52}) from $0$ to $t$ and combining the estimates for $J_{1}\sim J_{6}$, we get
\begin{align}\label{2.67}
\mathcal{E}_{11}(t)&=\sup_{0\leq t'\leq t}\left[(1+t')^{1+2s_{1}-s_{0}}(\|u(t')\|_{\dot{H}^{1+s_{1}-s_{0}}\cap\dot{H}^{2}}^{2}+\|\mathbb{P}\diver  G(t')\|_{\dot{H}^{s_{1}-s_{0}}\cap\dot{H}^{1}}^{2})\right]\nonumber \\
&\ \ \ \  \int_{0}^{t}(1+t')^{1+2s_{1}-s_{0}}\|\nabla u(t')\|_{\dot{H}^{1+s_{1}-s_{0}}\cap\dot{H}^{2}}^{2}\,dt'\nonumber \\
&\lesssim\mathcal{E}_{1}(0)+\mathcal{E}_{0}(t)+\mathcal{E}_{1}(t)+\mathcal{E}_{0}^{\frac{3}{2}}(t)+\mathcal{E}_{1}^{\frac{3}{2}}(t)+\mathcal{E}_{a}^{\frac{3}{2}}(t).
\end{align}

\textbf{Step 2.} The estimate of $\mathcal{E}_{12}(t)$.\\
Applying $\nabla^{k}\mathbb{P}(k=1+s_{1}-s_{0},1)$ on the first equation of system (\ref{2.4}), we get
\begin{equation}\label{2.68}
\nabla^{k} u_{t}-\nabla^{k}\Delta u+\nabla^{k}\mathbb{P}(u\cdot\nabla u)+\nabla^{k}\mathbb{P}\bigg[\frac{\rho}{\rho+1}(\Delta u+\diver  G)\bigg]=\nabla^{k}\mathbb{P}\diver G.
\end{equation}
Taking the inner product of (\ref{2.68}) with $\nabla^{k}\mathbb{P}\diver G$, and multiplying the time weight $(1+t')^{1+2s_{1}-s_{0}}$, one gets
\begin{equation}\label{2.69}
(1+t')^{1+2s_{1}-s_{0}}\|\mathbb{P}\diver  G\|_{\dot{H}^{1+s_{1}-s_{0}}\cap\dot{H}^{1}}^{2}=J_{7}+J_{8}+J_{9}+J_{10},
\end{equation}
where
\begin{align*}
&J_{7}=-(1+t')^{1+2s_{1}-s_{0}}\left\langle\Delta u,\mathbb{P}\diver  G\right\rangle_{\dot{H}^{1+s_{1}-s_{0}}\cap\dot{H}^{1}};\\
&J_{8}=(1+t')^{1+2s_{1}-s_{0}}\left\langle\mathbb{P}(u\cdot\nabla u),\mathbb{P}\diver  G \right\rangle_{\dot{H}^{1+s_{1}-s_{0}}\cap\dot{H}^{1}};\\
&J_{9}=(1+t')^{1+2s_{1}-s_{0}}\left\langle u_{t},\mathbb{P}\diver  G \right\rangle_{\dot{H}^{1+s_{1}-s_{0}}\cap\dot{H}^{1}};\\
&J_{10}=(1+t')^{1+2s_{1}-s_{0}}\left\langle\mathbb{P}[g(\rho)(\Delta u+\diver  G)],\mathbb{P}\diver  G\right\rangle_{\dot{H}^{1+s_{1}-s_{0}}\cap\dot{H}^{1}}.
\end{align*}
For $J_{7}$, it is easy to see that
\begin{align}\label{2.70}
\int_{0}^{t}|J_{7}(t')|\,dt'&\lesssim\int_{0}^{t}(1+t')^{1+2s_{1}-s_{0}}\|\Delta u\|_{\dot{H}^{1+s_{1}-s_{0}}\cap\dot{H}^{1}}\|\mathbb{P}\diver  G\|_{\dot{H}^{1+s_{1}-s_{0}}\cap\dot{H}^{1}}\,dt'\nonumber \\
&\lesssim\int_{0}^{t}(1+t')^{1+2s_{1}-s_{0}}\|\nabla u\|_{\dot{H}^{1+s_{1}-s_{0}}\cap\dot{H}^{2}}\|\mathbb{P}\diver  G\|_{\dot{H}^{1+s_{1}-s_{0}}\cap\dot{H}^{1}}\,dt'\nonumber \\
&\lesssim\mathcal{E}_{11}(t)+\mathcal{E}_{12}(t).
\end{align}
For $J_{8}$, using integration by parts, H\"{o}lder's inequality, Lemma \ref{le5.1} and Lemma \ref{le5.4}, we arrive at
\begin{align}\label{2.71}
|J_{8}(t')|&\lesssim(1+t')^{1+2s_{1}-s_{0}}\big(\|\nabla u\cdot\nabla u\|_{\dot{H}^{s_{1}-s_{0}}\cap L^{2}}+\|u\cdot\nabla\nabla u\|_{\dot{H}^{s_{1}-s_{0}}\cap L^{2}}\big)\|\mathbb{P}\diver  G\|_{\dot{H}^{1+s_{1}-s_{0}}\cap\dot{H}^{1}}\nonumber \\
&\lesssim(1+t')^{1+2s_{1}-s_{0}}\big(\|\nabla^{2} u\|_{\dot{H}^{-s_{0}+s_{1}}\cap \dot{H}^{s_{0}-s_{1}}}\|\nabla u\|_{\dot{H}^{s_{1}-s_{0}}}+\|\nabla u\|_{L^{4}}\|\nabla u\|_{L^{4}}\nonumber \\
&\ \ \ \ +\|\nabla u\|_{\dot{H}^{-s_{0}+s_{1}}\cap \dot{H}^{s_{0}-s_{1}}}\|\nabla^{2} u\|_{\dot{H}^{s_{1}-s_{0}}}+\|u\|_{L^{\infty}}\|\nabla^{2} u\|_{L^{2}}\big)\|\mathbb{P}\diver  G\|_{\dot{H}^{1+s_{1}-s_{0}}\cap\dot{H}^{1}}\nonumber \\
&\lesssim(1+t')^{1+2s_{1}-s_{0}}\big(\|\nabla u\|_{\dot{H}^{1+s_{1}-s_{0}}\cap \dot{H}^{2}}\|u\|_{\dot{H}^{1+s_{1}-s_{0}}}+\|u\|_{\dot{H}^{1+s_{1}-s_{0}}\cap \dot{H}^{2}}^{2}\nonumber \\
&\ \ \ \ +\|u\|_{\dot{H}^{1+s_{1}-s_{0}}\cap \dot{H}^{2}}\|\nabla u\|_{\dot{H}^{1+s_{1}-s_{0}}}\big)\|\mathbb{P}\diver  G\|_{\dot{H}^{1+s_{1}-s_{0}}\cap\dot{H}^{1}}.
\end{align}
Integrating (\ref{2.71}) from $0$ to $t$, and using H\"{o}lder's inequality, one gets
\begin{align}\label{2.72}
\int_{0}^{t}|J_{8}(t')|\,dt'\lesssim\mathcal{E}_{11}^{\frac{3}{2}}(t)+\mathcal{E}_{11}^{\frac{1}{2}}(t)\mathcal{E}_{12}(t)
\lesssim\mathcal{E}_{11}^{\frac{3}{2}}(t)+\mathcal{E}_{12}^{\frac{3}{2}}(t).
\end{align}
For $J_{9}$, utilize integration by parts and rewrite $J_{9}=J_{91}+J_{92}+J_{93}$, where
\begin{align*}
&J_{91}=\frac{d}{dt}\bigg[(1+t')^{1+2s_{1}-s_{0}}\left\langle u,\mathbb{P}\diver  G \right\rangle_{\dot{H}^{1+s_{1}-s_{0}}\cap\dot{H}^{1}}\bigg];\\
&J_{92}=-(1+2s_{1}-s_{0})(1+t')^{2s_{1}-s_{0}}\left\langle u,\mathbb{P}\diver  G \right\rangle_{\dot{H}^{1+s_{1}-s_{0}}\cap\dot{H}^{1}};\\
&J_{93}=-(1+t')^{1+2s_{1}-s_{0}}\left\langle u,\mathbb{P}\diver  G_{t}\right\rangle_{\dot{H}^{1+s_{1}-s_{0}}\cap\dot{H}^{1}}.
\end{align*}
Integrating $J_{91},J_{92}$ from $0$ to $t$, we get
\begin{align}\label{2.73}
\int_{0}^{t}|J_{91}(t')|\,dt'&\lesssim\sup_{0\leq t'\leq t}(1+t')^{1+2s_{1}-s_{0}}\big(\|\nabla^{2}u\|_{\dot{H}^{s_{1}-s_{0}}}\|\mathbb{P}\diver  G\|_{\dot{H}^{1+s_{1}-s_{0}}}+\|\nabla u\|_{L^{2}}\|\mathbb{P}\diver  G\|_{\dot{H}^{1}}\big)\nonumber \\
&\lesssim\mathcal{E}_{11}(t).
\end{align}
\begin{align}\label{2.74}
\int_{0}^{t}|J_{92}(t')|\,dt'&\lesssim\int_{0}^{t}(1+t')^{-1}(1+t')^{\frac{1+2s_{1}-s_{0}}{2}}\|u\|_{\dot{H}^{1+s_{1}-s_{0}}\cap\dot{H}^{1}}(1+t')^{\frac{1+2s_{1}-s_{0}}{2}}\|\mathbb{P}\diver  G\|_{\dot{H}^{1+s_{1}-s_{0}}\cap\dot{H}^{1}}\,dt'\nonumber \\
&\lesssim\mathcal{E}_{11}^{\frac{1}{2}}(t)\mathcal{E}_{12}^{\frac{1}{2}}(t)\lesssim\mathcal{E}_{11}(t)+\mathcal{E}_{12}(t).
\end{align}
For the last term $J_{93}$, applying projection operator $\mathbb{P}$ on the second equation of system (\ref{2.4}), we obtain
\begin{align}\label{2.75}
J_{93}&=-(1+t')^{1+2s_{1}-s_{0}}\left\langle u,\Delta u-\mathbb{P}\diver  (u\cdot\nabla G)-\mathbb{P}\diver Q(\nabla u, G)\right\rangle_{\dot{H}^{1+s_{1}-s_{0}}\cap\dot{H}^{1}}\nonumber \\
&:=J_{931}+J_{932}+J_{933}.
\end{align}
We can directly deal with the term $J_{931}$ as follows:
\begin{align}\label{2.76}
\int_{0}^{t}|J_{931}(t')|\,dt'&\lesssim\int_{0}^{t}(1+t')^{1+2s_{1}-s_{0}}\|\nabla u\|_{\dot{H}^{1+s_{1}-s_{0}}\cap\dot{H}^{1}}^{2}\,dt'\nonumber \\
&\lesssim\mathcal{E}_{11}(t).
\end{align}
For $J_{932}$, utilizing Lemma \ref{le2.2} and Lemma \ref{le5.4}, and considering the boundedness of Riesz operator $\mathcal{R}_{i}=(-\Delta)^{-\frac{1}{2}}\nabla_{i}$ from $L^{2}$ into itself. By the same argument adopted in handing $J_{8}$, we have
\begin{align}\label{2.77}
J_{932}&=(1+t')^{1+2s_{1}-s_{0}}\left\langle u,\mathbb{P}(u\cdot\nabla\mathbb{P}\diver G)+\mathbb{P}(\nabla u\cdot\nabla G)-\mathbb{P}(\nabla u\cdot\nabla(-\Delta)^{-1}\diver\diver G)\right\rangle_{\dot{H}^{1+s_{1}-s_{0}}\cap\dot{H}^{1}}\nonumber \\
&\lesssim(1+t')^{1+2s_{1}-s_{0}}\|\nabla u\|_{\dot{H}^{1+s_{1}-s_{0}}\cap\dot{H}^{1}}\big(\|u\cdot\nabla\mathbb{P}\diver G\|_{\dot{H}^{s_{1}-s_{0}}\cap L^{2}}+\|\nabla u\cdot\nabla G\|_{\dot{H}^{s_{1}-s_{0}}\cap L^{2}}\nonumber \\
&\ \  \  \ +\|\nabla u\cdot\nabla(-\Delta)^{-1}\diver\diver G\|_{\dot{H}^{s_{1}-s_{0}}\cap L^{2}}\big)\nonumber \\
&\lesssim(1+t')^{1+2s_{1}-s_{0}}\|\nabla u\|_{\dot{H}^{1+s_{1}-s_{0}}\cap\dot{H}^{1}}\big(\|\nabla u\|_{\dot{H}^{-s_{0}+s_{1}}\cap \dot{H}^{s_{0}-s_{1}}}\|\nabla\mathbb{P}\diver G\|_{\dot{H}^{s_{1}-s_{0}}\cap L^{2}}\nonumber \\
&\ \  \  \ +\|\nabla^{2} u\|_{\dot{H}^{-s_{0}+s_{1}}\cap \dot{H}^{s_{0}-s_{1}}}\|\nabla G\|_{\dot{H}^{s_{1}-s_{0}}\cap L^{2}}\big).
\end{align}
Integrating (\ref{2.77}) from $0$ to $t$, we get
\begin{align}\label{2.78}
\int_{0}^{t}|J_{932}(t')|\,dt'&\lesssim\sup_{0\leq t'\leq t}\|(u,G)\|_{\dot{H}^{-s_{1}}\cap\dot{H}^{2}}\int_{0}^{t}(1+t')^{1+2s_{1}-s_{0}}\|\nabla u\|_{\dot{H}^{1+s_{1}-s_{0}}\cap\dot{H}^{1}}\big(\|\nabla\mathbb{P}\diver G\|_{\dot{H}^{s_{1}-s_{0}}\cap L^{2}}\nonumber \\
&\ \  \  \ +\|\nabla^{2} u\|_{\dot{H}^{-s_{0}+s_{1}}\cap \dot{H}^{s_{0}-s_{1}}}\big)\,dt'\nonumber \\
&\lesssim\mathcal{E}_{0}^{\frac{1}{2}}(t)\mathcal{E}_{1}(t)\lesssim\mathcal{E}_{0}^{\frac{3}{2}}(t)+\mathcal{E}_{1}^{\frac{3}{2}}(t).
\end{align}
Using integration by parts, H\"{o}lder's inequality, Lemma \ref{le5.1} and Lemma \ref{le5.4}, we can deal with the term $J_{933}$ as follows:
\begin{align}\label{2.79}
|J_{933}(t')|&\lesssim(1+t')^{1+2s_{1}-s_{0}}\|\nabla u\|_{\dot{H}^{1+s_{1}-s_{0}}\cap\dot{H}^{1}}\|Q(\nabla u, G)\|_{\dot{H}^{1+s_{1}-s_{0}}\cap\dot{H}^{1}}\nonumber \\
&\lesssim(1+t')^{1+2s_{1}-s_{0}}\|\nabla u\|_{\dot{H}^{1+s_{1}-s_{0}}\cap\dot{H}^{1}}\big(\|Q(\nabla u, \nabla G)\|_{\dot{H}^{s_{1}-s_{0}}\cap L^{2}}+\|Q(\nabla\nabla u, G)\|_{\dot{H}^{s_{1}-s_{0}}\cap L^{2}}\big)\nonumber \\
&\lesssim(1+t')^{1+2s_{1}-s_{0}}\|\nabla u\|_{\dot{H}^{1+s_{1}-s_{0}}\cap\dot{H}^{1}}\big(\|\nabla^{2} u\|_{\dot{H}^{-s_{0}+s_{1}}\cap \dot{H}^{s_{0}-s_{1}}}\|\nabla G\|_{\dot{H}^{s_{1}-s_{0}}}+\|\nabla u\|_{L^{\infty}}\|\nabla G\|_{L^{2}}\nonumber \\
&\ \ \ \ +\|\nabla G\|_{\dot{H}^{-s_{0}+s_{1}}\cap \dot{H}^{s_{0}-s_{1}}}\|\nabla^{2} u\|_{\dot{H}^{s_{1}-s_{0}}}+\|G\|_{L^{\infty}}\|\nabla^{2} u\|_{L^{2}}\big)\nonumber \\
&\lesssim(1+t')^{1+2s_{1}-s_{0}}\|\nabla u\|_{\dot{H}^{1+s_{1}-s_{0}}\cap\dot{H}^{1}}\big(\|\nabla^{2} u\|_{\dot{H}^{-s_{0}+s_{1}}\cap \dot{H}^{s_{0}-s_{1}}}\|\nabla G\|_{\dot{H}^{s_{1}-s_{0}}\cap L^{2}}\nonumber \\
&\ \ \ \ +\|\nabla G\|_{\dot{H}^{-s_{0}+s_{1}}\cap \dot{H}^{s_{0}-s_{1}}}\|\nabla^{2} u\|_{\dot{H}^{s_{1}-s_{0}}\cap L^{2}}\big).
\end{align}
Integrating (\ref{2.79}) from $0$ to $t$, we get
\begin{align}\label{2.80}
\int_{0}^{t}|J_{933}(t')|\,dt'&\lesssim\sup_{0\leq t'\leq t}\|G\|_{\dot{H}^{-s_{1}}\cap\dot{H}^{2}}\int_{0}^{t}(1+t')^{1+2s_{1}-s_{0}}\|\nabla u\|_{\dot{H}^{1+s_{1}-s_{0}}\cap\dot{H}^{2}}^{2}\,dt'\nonumber \\
&\lesssim\mathcal{E}_{0}^{\frac{1}{2}}(t)\mathcal{E}_{1}(t)\lesssim\mathcal{E}_{0}^{\frac{3}{2}}(t)+\mathcal{E}_{1}^{\frac{3}{2}}(t).
\end{align}
Combining the estimates for $J_{91}\sim J_{93}$ and integrating $J_{9}$ from $0$ to $t$, we get
\begin{align}\label{2.81}
\int_{0}^{t}|J_{9}(t')|\,dt'\lesssim\mathcal{E}_{0}^{\frac{3}{2}}(t)+\mathcal{E}_{1}^{\frac{3}{2}}(t)+\mathcal{E}_{1}(t).
\end{align}
For $J_{10}$, using integration by parts, H\"{o}lder's inequality, Lemma \ref{le5.1} and Lemma \ref{le5.4}, we have
\begin{align}\label{2.82}
|J_{10}(t')|&\lesssim(1+t')^{1+2s_{1}-s_{0}}\big(\|\nabla(g(\rho))(\Delta u+\diver  G)\|_{\dot{H}^{s_{1}-s_{0}}\cap L^{2}}\nonumber \\
&\ \ \ \ +\|g(\rho)\nabla(\Delta u+\diver  G)\|_{\dot{H}^{s_{1}-s_{0}}\cap L^{2}}\big)\|\mathbb{P}\diver  G\|_{\dot{H}^{1+s_{1}-s_{0}}\cap\dot{H}^{1}}\nonumber \\
&\lesssim(1+t')^{1+2s_{1}-s_{0}}\big(\|\nabla(g(\rho))\Delta u\|_{\dot{H}^{s_{1}-s_{0}}\cap L^{2}}+\|\nabla(g(\rho))\diver  G\|_{\dot{H}^{s_{1}-s_{0}}\cap L^{2}}+\|g(\rho)\nabla\Delta u\|_{\dot{H}^{s_{1}-s_{0}}\cap L^{2}}\nonumber \\
&\ \ \ \ +\|g(\rho)\nabla\diver  G\|_{\dot{H}^{s_{1}-s_{0}}\cap L^{2}}\big)\|\mathbb{P}\diver  G\|_{\dot{H}^{1+s_{1}-s_{0}}\cap\dot{H}^{1}}\nonumber \\
&\lesssim(1+t')^{1+2s_{1}-s_{0}}\big(\|\vert\nabla\vert^{\sigma}\Delta u\|_{\dot{H}^{-r_{1}}\cap \dot{H}^{r_{1}}}\|\nabla(g(\rho))\|_{\dot{H}^{s_{1}-s_{0}}}\nonumber \\
&\ \ \ \ +\|\nabla(g(\rho))\|_{L^{3}}\|\Delta u\|_{L^{6}}+\|\vert\nabla\vert^{\sigma}\mathbb{P}\diver  G\|_{\dot{H}^{-r_{1}}\cap \dot{H}^{r_{1}}}\|\nabla(g(\rho))\|_{\dot{H}^{s_{1}-s_{0}}}\nonumber\\
&\ \ \ \ +\|\nabla(g(\rho))\|_{L^{2}}\|\mathbb{P}\diver  G\|_{L^{\infty}}+\|g(\rho)\|_{L^{\infty}}\|(\nabla\Delta  u,\nabla\mathbb{P}\diver  G)\|_{L^{2}}\nonumber \\
&\ \ \ \ +\|\nabla(g(\rho))\|_{\dot{H}^{-s_{0}+s_{1}}\cap \dot{H}^{s_{0}-s_{1}}}\|(\nabla\Delta  u,\nabla\mathbb{P}\diver  G)\|_{\dot{H}^{s_{1}-s_{0}}}\big)\|\mathbb{P}\diver  G\|_{\dot{H}^{1+s_{1}-s_{0}}\cap\dot{H}^{1}}.
\end{align}
Integrating (\ref{2.82}) from $0$ to $t$, we get
\begin{align}\label{2.83}
\int_{0}^{t}&|J_{10}(t')|\,dt'\nonumber\\
&\lesssim\sup_{0\leq t'\leq t}\|\rho\|_{H^{2}}\int_{0}^{t}(1+t')^{1+2s_{1}-s_{0}}\big(\|\nabla u \|_{\dot{H}^{1+s_{1}-s_{0}}\cap\dot{H}^{2}}+\|\mathbb{P}\diver  G\|_{\dot{H}^{1+s_{1}-s_{0}}\cap\dot{H}^{1}}\big)\|\mathbb{P}\diver  G\|_{\dot{H}^{1+s_{1}-s_{0}}\cap\dot{H}^{1}}\,dt'\nonumber \\
&\lesssim\mathcal{E}_{a}^{\frac{1}{2}}(t)\mathcal{E}_{1}(t)\lesssim\mathcal{E}_{a}^{\frac{3}{2}}(t)+\mathcal{E}_{1}^{\frac{3}{2}}(t).
\end{align}
Combining the estimates of (\ref{2.70}), (\ref{2.72}), (\ref{2.81}), and (\ref{2.83}) together, and integrating (\ref{2.69}) from $0$ to $t$, we have
\begin{align}\label{2.84}
\mathcal{E}_{12}(t)&=\int_{0}^{t}(1+t')^{1+2s_{1}-s_{0}}\|\mathbb{P}\diver  G(t')\|_{\dot{H}^{1+s_{1}-s_{0}}\cap\dot{H}^{1}}^{2}\,dt'\nonumber \\
&\lesssim\mathcal{E}_{1}(t)+\mathcal{E}_{0}^{\frac{3}{2}}(t)+\mathcal{E}_{1}^{\frac{3}{2}}(t)+\mathcal{E}_{a}^{\frac{3}{2}}(t).
\end{align}

Finally, combining the estimates of (\ref{2.67}) and (\ref{2.84}) together, we can conclude that
\begin{align}\label{2.85}
\mathcal{E}_{1}(t)&\lesssim\mathcal{E}_{1}(0)+\mathcal{E}_{0}(t)+\mathcal{E}_{1}(t)+\mathcal{E}_{0}^{\frac{3}{2}}(t)+\mathcal{E}_{1}^{\frac{3}{2}}(t)+\mathcal{E}_{a}^{\frac{3}{2}}(t).
\end{align}

\end{proof}

\section{ Proof of Theorem 1.1 }\label{se3}
In this section, we first give the local well-posedness of the solution of the Cauchy problem \eqref{1.1}--\eqref{1.1'}, which can be established by the standard method in \cite{Chemin-Masmoudi2001}.

\begin{Proposition}
Assume that the initial data $(u_{0},G_{0})\in \dot{H}^{-s_{1}}(\mathbb{R}^{2})\cap\dot{H}^{2}(\mathbb{R}^{2}), (\tilde{\rho}_{0}-1)\in \dot{H}^{s_{1}-s_{0}}(\mathbb{R}^{2})\cap H^{2}(\mathbb{R}^{2}), (\mathbb{F}_{0}-\mathbb{I})\in H^{2}(\mathbb{R}^{2})$. Then there exists a constant $T_{1}>0$ such that the Cauchy problem \eqref{1.1}--\eqref{1.1'} possesses a unique solution $(u, \tilde{\rho}, \mathbb{F})$ satisfying
\begin{equation*}
\Big(u,(\tilde{\rho}-1),(\mathbb{F}-\mathbb{I})\Big)\in C\Big(0,T_{1};(\dot{H}^{-s_{1}}\cap\dot{H}^{2})\times (\dot{H}^{s_{1}-s_{0}}\cap H^{2})\times H^{2}\Big).
\end{equation*}
\end{Proposition}

With the help of the global priori estimates in section \ref{se2} and Proposition 3.1, we are devoted to the proof of Theorem 1.1.
Considering the total energy $\mathcal{E}_{total}(t)=\mathcal{E}_{0}(t)+\mathcal{E}_{1}(t)+\mathcal{E}_{a}(t)$, the estimates of Lemma \ref{le2.3}, Lemma \ref{le2.4} and Lemma \ref{le2.5}, employing Young's inequality and Lemma \ref{le5.1}, we can choose some constant $C_{0}$ such that
\begin{equation}\label{4.1}
\mathcal{E}_{total}(t)\leq C_{0}\mathcal{E}_{total}(0)+C_{0}\mathcal{E}_{total}(t)+C_{0}\mathcal{E}_{total}^{\frac{3}{2}}(t)+C_{0}\mathcal{E}_{total}^{\frac{9}{4}}(t).
\end{equation}

Applying the assumptions of initial data in Theorem \ref{th1.1}, we can choose the constant $\varepsilon>0$ to be sufficiently small so that
\begin{equation*}
C_{0}\mathcal{E}_{total}(0)\leq\frac{\varepsilon}{2}.
\end{equation*}
From the existence of local solution in Proposition \ref{4.1} and the standard energy method, there exists a time $T>0$ such that
\begin{equation}\label{4.2}
\mathcal{E}_{total}(t)\leq \varepsilon,\ \ \  \forall \ \ t\in[0,T].
\end{equation}
Let $T_{max}$ is the lifespan of solutions to (\ref{1.1}) by
\begin{equation*}
T_{max}:=\sup\left\{t: \sup_{0\leq s\leq t}\mathcal{E}_{total}(t)\leq \varepsilon\right\}.
\end{equation*}
Combining the continuation argument and $\varepsilon$ is small enough, from (\ref{4.1}), we can conclude that $T_{max}=\infty$.
Therefore, the proof of Theorem 1.1 is completed.

\appendix
\renewcommand{\appendixname}{Appendix~\Alph{section}}

\section{Tools}\label{appendix}

In the appendix, we state some useful results that have been frequently used in the previous sections.

The following is the general Gagliardo-Nirenberg inequality:

\begin{Lemma}\label{le5.1}
Let $1\leq q, r\leq \infty$ and $0\leq m\leq\alpha< l$, then we have
\begin{equation}\label{5.1}
\|\nabla^{\alpha}f\|_{L^p}\lesssim \|\nabla^{m}f\|_{L^q}^{1-\theta}\|\nabla^{l}f\|_{L^r}^{\theta},
\end{equation}
where $\frac{\alpha}{l}\leq\theta\leq1$ and $\alpha$ satisfies
\begin{equation*}
  \frac{\alpha}{2}-\frac1p=\left(\frac{m}{2}-\frac{1}{q}\right)(1-\theta)+\left(\frac{l}{2}-\frac{1}{r}\right)\theta
\end{equation*}
with the following exceptional cases: \\
$(1)$ If $\alpha=0, rl<2, q=\infty$, then one needs the additional assumption that either $f$ tends to zero at infinity or $f\in L^{\tilde{q}}$ for some finite $\tilde{q}>0$;\\
$(2)$ If $1<r<\infty$, and $l-\alpha-\frac{2}{r}$ is a non negative integer, then the interpolation inequalities in (\ref{5.1}) hold only for $\theta$ satisfying $\frac{\alpha}{l}\leq \theta<1$;\\
$(3)$ The interpolation inequalities in (\ref{5.1}) also hold for fractional derivatives.
\end{Lemma}

\begin{proof}
See Theorem (pp. 125--126) in \cite{Nirenberg1959}.
\end{proof}

\begin{Lemma}\label{le2.1}
If any smooth function $g(\cdot)$ is defined around $0$ with $g(0)=0$, which satisfies
\begin{equation*}
\text{$g(\rho)\sim\rho$\ \ \  and  \ \ \  $\|g^{(k)}(\rho)\|_{L^{2}}\leq C(k)$\ \ \ \ \ \  for any \ \  $0\leq k \leq2$ },
\end{equation*}
then it holds that
\begin{align*}
&\|g(\rho)\|_{L^{p}}\lesssim\|\rho\|_{L^{p}}, \ \ \ \text{for some~$p$ with $1\leq p\leq\infty$} , \\
&\|\nabla^{k}g(\rho)\|_{L^{p}}\lesssim\|\nabla^{k}\rho\|_{L^{p}}, \ \ \ k=1,2.
\end{align*}
\end{Lemma}
\begin{proof}
See Proposition 2.2 in \cite{Zhu2022}.
\end{proof}

In our arguments, we also need to use the following fractional Leibniz rule and the estimates for commutators related to fractional negative derivatives:
\begin{Lemma}\label{le5.2}
Let $s>0$, $n\geq2$, $1<p<\infty$, and $1\leq p_{1},q_{2}\leq\infty$. Then it holds that
\begin{equation*}
\|\,\vert\nabla\vert^{s}(fg)\|_{L^{p}(\mathbb{R}^{n})}
\lesssim\|\,\vert\nabla\vert^{s}f\|_{L^{p_{1}}(\mathbb{R}^{n})}
\|g\|_{L^{q_{1}}(\mathbb{R}^{n})}+\|f\|_{L^{p_{2}}(\mathbb{R}^{n})}
\|\vert\nabla\vert^{s}g\|_{L^{q_{2}}(\mathbb{R}^{n})},
\end{equation*}
where
\begin{align*}
\frac{1}{p}=\frac{1}{p_{1}}+\frac{1}{q_{1}}=\frac{1}{p_{2}}+\frac{1}{q_{2}}.
\end{align*}
\end{Lemma}
\begin{proof}
The proof of this estimate can be consulted in \cite{Kato-Ponce1988}.
\end{proof}

\begin{Lemma}\label{le5.3}
Let $r_{0}$ be a real number with $-1<r_{0}<0$. For any divergence free vector function $u\in L^{2}(\mathbb{R}^{2})$ and function $v\in L^{2}(\mathbb{R}^{2})$, it holds that
\begin{equation*}
\|[\vert\nabla\vert^{r_{0}},u\cdot\nabla]v\|_{L^{2}}\lesssim\|\nabla^{2}u\|_{\dot{H}^{-r_{1}}\cap\dot{H}^{r_{1}}}\|v\|_{\dot{H}^{r_{0}}},
\end{equation*}
where $r_{1}$ is an arbitrarily positive number.
\end{Lemma}
\begin{proof}
See Proposition 2.3 in \cite{Chen-Zhu2023}.
\end{proof}

Moreover, we also need to use the following bilinear estimate related to fractional negative derivatives:
\begin{Lemma}\label{le5.4}
Let $r_{0}$ be a real number with $-1<r_{0}<0$. For any functions $f,g\in L^{2}(\mathbb{R}^{2})$, we have
\begin{equation*}
  \|\,|\nabla\vert^{r_{0}}(fg)\|_{L^2}\leq \|\,|\nabla\vert^{\sigma}f\|_{\dot{H}^{-r_{1}}\cap\dot{H}^{r_{1}}}\|g\|_{\dot{H}^{r_{0}}},
\end{equation*}
if $\sigma$ and $r_1$ satisfy one of the following two assumptions:
\begin{enumerate}
  \item $\sigma=1,\ r_1>0$;
  \item $\max\{\frac{2}{3},2-4r_{1},-r_{0}+r_{1}\}<\sigma<1,\ \frac{1}{4}<r_{1}<-r_{0}$.
\end{enumerate}
\end{Lemma}
\begin{proof}
For the case $(1)$, it has been proved in \cite{Chen-Zhu2023}. Next, it suffices to prove the case $(2)$. By using the Fourier transform, we can obtain
\begin{align}\label{5.2}
\mathscr{F}(\vert\nabla\vert^{r_{0}}(fg))(\xi)&=\int_{\mathbb{R}^{2}}|\xi|^{r_{0}}\hat{f}(\xi-\eta)\hat{g}(\eta)\,d\eta \nonumber \\
&=\int_{|\eta|\leq2|\xi|}|\xi|^{r_{0}}\hat{f}(\xi-\eta)\hat{g}(\eta)\,d\eta+\int_{|\eta|>2|\xi|}|\xi|^{r_{0}}\hat{f}(\xi-\eta)\hat{g}(\eta)\,d\eta \nonumber \\
&:=Q_{1}+Q_{2}.
\end{align}
For the case $|\eta|\leq2|\xi|$ and $r_{0}<0$, we have
\begin{equation}\label{5.3}
Q_{1}\lesssim\int_{|\eta|\leq2|\xi|}|\hat{f}(\xi-\eta)||\eta|^{r_{0}}|\hat{g}(\eta)|\,d\eta
\end{equation}
and
\begin{align}\label{5.4}
\|Q_{1}\|_{L^{2}}&\lesssim\|\hat{f}\|_{L^{1}}\|g\|_{\dot{H}^{r_{0}}} \nonumber \\
&\lesssim\bigg(\int_{|\xi|\leq1}\frac{1}{|\xi|^{\sigma-r_{1}}}|\xi|^{\sigma-r_{1}}|\hat{f}(\xi)|\,d\xi+\int_{|\xi|>1}\frac{1}{|\xi|^{\frac{3}{4}\sigma}}
\frac{1}{|\xi|^{\frac{1}{4}\sigma+r_{1}}}|\xi|^{\sigma+r_{1}}|\hat{f}(\xi)|\,d\xi\bigg)\|g\|_{\dot{H}^{r_{0}}} \nonumber \\
&\lesssim\big(\|\,\vert\nabla\vert^{\sigma-r_{1}}\hat{f}\|_{L^{2}}+\|\,\vert\nabla\vert^{\sigma+r_{1}}\hat{f}\|_{L^{2}}\big)\|g\|_{\dot{H}^{r_{0}}},
\end{align}
where $\max\{\frac{2}{3},2-4r_{1}\}<\sigma<1,\frac{1}{4}<r_{1}<-r_{0}$.\\
Next, for the case $|\eta|>2|\xi|$, we have $|\xi-\eta|\geq\frac{|\xi|}{2}$. Then it holds that
\begin{align}\label{5.5}
Q_{2}&\lesssim|\xi|^{r_{0}}\int_{|\eta|>2|\xi|}|\xi-\eta|^{-r_{0}}|\hat{f}(\xi-\eta)||\eta|^{r_{0}}|\hat{g}(\eta)|\,d\eta \nonumber \\
&\lesssim|\xi|^{r_{0}}H(f,g),
\end{align}
where $H(f,g):=\int_{\mathbb{R}^{2}}|\xi-\eta|^{-r_{0}}|\hat{f}(\xi-\eta)||\eta|^{r_{0}}|\hat{g}(\eta)|\,d\eta$.
\begin{align}\label{5.6}
\|Q_{2}\|_{L^{2}}^{2}&\lesssim\int_{|\xi|\leq1}|\xi|^{2r_{0}}|H|^{2}\,d\xi+\int_{|\xi|>1}|\xi|^{2r_{0}}|H|^{2}\,d\xi \nonumber \\
&\lesssim\|H\|_{L^{p_{0}}}^{2}+\|H\|_{L^{p_{1}}}^{2},
\end{align}
where $p_{0}, p_{1}$ are required to satisfy $2<\frac{2}{1+r_{0}}<p_{0}, 2<p_{1}<\frac{2}{1+r_{0}}$.\\
\begin{align}\label{5.7}
\|H\|_{L^{p_{0}}\cap L^{p_{1}}}&\lesssim\|\mathscr{F}(\vert\nabla\vert^{-r_{0}}f)(\xi)\|_{L^{\frac{2p_{0}}{p_{0}+2}}\cap L^{\frac{2p_{1}}{p_{1}+2}}}\|\mathscr{F}(\vert\nabla\vert^{r_{0}}g)(\xi)\|_{L^{2}} \nonumber \\
&\lesssim\|\mathscr{F}(\vert\nabla\vert^{-r_{0}}f)(\xi)\|_{L^{\frac{2p_{0}}{p_{0}+2}}\cap L^{\frac{2p_{1}}{p_{1}+2}}}\|\,\vert\nabla\vert^{r_{0}}g\|_{L^{2}},
\end{align}
Choosing $p_{0}, \sigma$ with $\frac{2}{\sigma+r_{0}+r_{1}}<p_{0}<\frac{2}{\sigma+r_{0}-r_{1}}, -r_{0}+r_{1}\leq\sigma<1$, we obtain
\begin{align}\label{5.8}
\|\mathscr{F}(\vert\nabla\vert^{-r_{0}}f)(\xi)\|_{L^{\frac{2p_{0}}{p_{0}+2}}}&=\|\,\vert\xi\vert^{-\sigma-r_{0}}\mathscr{F}(\vert\nabla\vert^{\sigma}f)(\xi)\|_{L^{\frac{2p_{0}}{p_{0}+2}}}\nonumber \\
&\lesssim\|\,\vert\xi\vert^{-\sigma-r_{0}+r_{1}}\mathscr{F}(\vert\nabla\vert^{\sigma-r_{1}}f)(\xi)\|_{L^{\frac{2p_{0}}{p_{0}+2}}(|\xi|\leq1)}\nonumber \\
&\ \ \ \ +\|\,\vert\xi\vert^{-\sigma-r_{0}-r_{1}}\mathscr{F}(\vert\nabla\vert^{\sigma+r_{1}}f)(\xi)\|_{L^{\frac{2p_{0}}{p_{0}+2}}(|\xi|>1)}\nonumber \\
&\lesssim\|\,\vert\xi\vert^{-\sigma-r_{0}+r_{1}}\|_{L^{p_{0}}(|\xi|\leq1)}\|\,\vert\nabla\vert^{\sigma-r_{1}}f\|_{L^{2}}\nonumber \\
&\ \ \ \ +\|\,\vert\xi\vert^{-\sigma-r_{0}-r_{1}}\|_{L^{p_{0}}(|\xi|>1)}\|\,\vert\nabla\vert^{\sigma+r_{1}}f\|_{L^{2}}\nonumber \\
&\lesssim\|\,\vert\nabla\vert^{\sigma}f\|_{\dot{H}^{-r_{1}}\cap\dot{H}^{r_{1}}}.
\end{align}
The estimate for $\|\mathscr{F}(\vert\nabla\vert^{-r_{0}}f)(\xi)\|_{L^{\frac{2p_{1}}{p_{1}+2}}}$ is also similar. We can choose $p_{1}, \sigma$ satisfying $\frac{2}{\sigma+r_{0}+r_{1}}<p_{1}<\frac{2}{\sigma+r_{0}-r_{1}}, -r_{0}+r_{1}\leq\sigma<1$. Then we can establish the estimate for $Q_{2}$. Together with (\ref{5.2}) and (\ref{5.4}), the proof of this proposition is achieved.

\end{proof}

Finally, we also provide the following useful regularity results for the Stokes problem:
\begin{Lemma}\label{le5.5}
Assume that $f\in L^{r}(\mathbb{R}^{n})$ with $2\leq r<\infty$. Let $(u,p)\in H^{1}(\mathbb{R}^{n})\times L^{2}(\mathbb{R}^{n})$ be the unique weak solution to the following Stokes problem
\begin{equation*}
\begin{cases}
-\mu\Delta u+\nabla p=f,\\
\diver u=0,\\
u(x)\rightarrow0, & x\rightarrow\infty.
\end{cases}
\end{equation*}
Then $(\nabla^{2}u,\nabla  p)\in L^{r}(\mathbb{R}^{n})$ and satisfies
\begin{equation*}
\|\nabla^{2}u\|_{L^{r}(\mathbb{R}^{n})}+\|\nabla p\|_{L^{r}(\mathbb{R}^{n})}
\lesssim \|f\|_{L^{r}(\mathbb{R}^{n})}.
\end{equation*}
\end{Lemma}
\begin{proof}
See Lemma 4.3 (p. 322) in \cite{Galdi2011}.
\end{proof}

\section*{Acknowledgements}

This work was partially supported by National Key R\&D Program of China (No. 2021YFA1002900), Guangzhou City Basic and Applied Basic Research Fund (No. 2024A04J6336), Yunnan Fundamental Research Projects (grant No. 202501AU070056), National Natural Science Foundation of China (No. 12161019), Guizhou Provincial Basic Research Program (Natural Science) ( No. QKHJC-ZK [2022] YB 318), Natural Science Research Project of Guizhou Provincial Department of Education (No. QJJ [2023] 011), and Academic Young Talent Fund of Guizhou Normal University (No. QSXM [2022] 03).

\bigskip

{\bf Data Availability:} Data sharing is not applicable to this article.

\bigskip

{\bf Conflict of interest:} This work does not have any conflicts of interest.

\end{document}